\documentclass[12pt]{amsart}

\calclayout
\allowdisplaybreaks

\usepackage{amssymb}
\usepackage[pdftex,colorlinks,citecolor=red]{hyperref}
\usepackage{graphicx,color}
\usepackage[table]{xcolor}
\usepackage{amsmath}
\usepackage{dsfont}
\usepackage{cleveref}
\usepackage{paralist}
\usepackage{tikz-cd}
\usepackage{tikz}
\usepackage{comment}
\usetikzlibrary{matrix}
\usepackage{url}
\usepackage{float}

\newtheorem{theorem}{Theorem}[section]
\newtheorem{proposition}[theorem]{Proposition}
\newtheorem{lemma}[theorem]{Lemma}
\newtheorem{corollary}[theorem]{Corollary}

\theoremstyle{definition}
\newtheorem{definition}[theorem]{Definition}

\newtheorem{example}[theorem]{Example}

\newtheorem{conjecture}[theorem]{Conjecture}
\newtheorem{remark}[theorem]{Remark}

\newcommand{\FF}{ \ensuremath{\mathbb{F}}}
\newcommand{\NN}{ \ensuremath{\mathbb{N}}}

\newcommand{\C}{\mathcal{C}}
\newcommand{\ST}{\mathcal{ST}}

\newcommand{\Lex}{\ensuremath{\mathrm{Lex}}}

\newcommand{\Skel}{\ensuremath{\mathrm{Skel}}}

\newcommand{\lex}{{\mathrm{lex}}}
\newcommand{\sqlex}{{\mathrm{sqlex}}}

\newcommand{\Mon}{{\mathrm{{Mon}}}}

\newcommand{\Tor}{\ensuremath{\mathrm{Tor}}\hspace{1pt}}

\makeatletter
\def\moverlay{\mathpalette\mov@rlay}
\def\mov@rlay#1#2{\leavevmode\vtop{%
   \baselineskip\z@skip \lineskiplimit-\maxdimen
   \ialign{\hfil$\m@th#1##$\hfil\cr#2\crcr}}}
\newcommand{\charfusion}[3][\mathord]{
    #1{\ifx#1\mathop\vphantom{#2}\fi
        \mathpalette\mov@rlay{#2\cr#3}
      }
    \ifx#1\mathop\expandafter\displaylimits\fi}
\makeatother
\newcommand{\bigcupdot}{\charfusion[\mathop]{\bigcup}{\cdot}}

\newcommand{\lk}{{\mathrm{lk}}}

\numberwithin{equation}{section}

\begin{document}
%\raggedbottom
\title{Graded Betti numbers of balanced simplicial complexes}%maybe change title

\author[M. Juhnke-Kubitzke]{Martina Juhnke-Kubitzke}
\email{juhnke-kubitzke@uni-osnabrueck.de}
\author[L. Venturello]{Lorenzo Venturello}
\email{lorenzo.venturello@uni-osnabrueck.de}
\address{
Universit\"{a}t Osnabr\"{u}ck,
Fakult\"{a}t f\"{u}r Mathematik,
Albrechtstra\ss e 28a,
49076 Osnabr\"{u}ck, GERMANY
}

\date{\today}

\thanks{
	Both authors were supported by the German Research Council DFG GRK-1916.
}
%Keyword and Subject Classes (if needed)

\begin{abstract}
We prove upper bounds for the graded Betti numbers of Stanley-Reisner rings of balanced simplicial complexes. Along the way we show bounds for Cohen-Macaulay graded rings $S/I$, where $S$ is a polynomial ring and $I\subseteq S$ is an homogeneous ideal containing a certain number of generators in degree 2, including the squares of the variables. Using similar techniques we provide upper bounds for the number of linear syzygies for Stanley-Reisner of balanced normal pseudomanifolds. Moreover, we compute explicitly the graded Betti numbers of cross-polytopal stacked spheres, and show that they only depend on the dimension and the number of vertices, rather than also the combinatorial type.
\end{abstract}

\maketitle
\section{Introduction}
In the last decades tremendous connections between combinatorics, topology and commutative algebra have been established. The theory of Cohen-Macaulay rings led to the proof of celebrated conjectures such as the upper bound theorem for spheres and the $g$-theorem for simplicial polytopes (see \cite{Stanley-greenBook} as a comprehensive reference). Since these results rely on algebraic properties of the Stanley-Reisner ring of simplicial complex, it is natural to investigate classical invariants of this ring, such as its minimal graded free resolution as a module over the polynomial ring. Our starting point are mainly two papers: In \cite{MN}, Migliore and Nagel showed upper bounds for the graded Betti numbers of simplicial polytopes. More recently, building on those results, Murai \cite{MUR} established a connection between a specific property of a triangulation, so-called \emph{tightness} and the graded Betti numbers of its Stanley-Reisner ring. Moreover, he employs upper bounds for graded Betti numbers to obtain a lower bound for the minimum number of vertices needed to triangulate a pseudomanifold with a given first (topological) Betti number. It is conceivable that for more specific classes of simplicial complexes, better bounds (for the graded Betti numbers) hold, which then can be turned again into lower bounds for the minimal number of vertices of such a simplicial complex. This serves as the motivation for this article, where we will focus on so-called \emph{balanced} simplicial complexes. \\
Those were originally introduced by Stanley \cite{St79} under the name \emph{completely balanced} as pure $(d-1)$-dimensional simplicial complexes whose vertex sets can be partitioned into $d$ classes, such that each class meets every face in at most (and hence exactly) one element. Following more recent papers, we will drop the word ``completely''  and we will not require balanced complexes to be pure. Notable examples are Coxeter complexes, Tits buildings as well as the order complex of a graded poset, with the vertex set partition given by the rank function. This last observation shows that the barycentric subdivision of any simplicial complex is balanced, which gives a constructive way of obtaining balanced triangulations of any topological space and shows that balancedness is a combinatorial rather than a topological property. In recent years, balanced simplicial complexes have been studied intensively and many classical results in face enumeration have been proven to possess a balanced analog (see e.g., \cite{GKN,JKM,KN,Juhnke:Murai:Novik:Sawaske}). 

The aim of this article it to continue with this line of research by studying graded Betti numbers of balanced simplicial complexes. Our main results establish upper bounds for different cases, including arbitrary balanced simplicial complexes, balanced Cohen-Macaulay complexes and balanced normal pseudomanifolds. Along the way, we derive upper bounds on the graded Betti numbers of homogeneous ideals with a high concentration of generators in degree $2$. 

The structure of this paper is the following:
\begin{itemize}
	\item \Cref{preliminaries} is devoted to the basic notions and definitions.
	\item In \Cref{balsimp} we use Hochster's formula to prove a first upper bound for the graded Betti numbers of an arbitrary balanced simplicial complexes (see \Cref{skeldpartite}).
	\item We next restrict ourselves to the Cohen-Macaulay case, and provide two different upper bounds in this setting. The first approach provides a bound for graded Betti numbers of ideals with a high concentration of generators in degree $2$, which immediately specializes to Stanley-Reisner ideals of balanced Cohen-Macaulay complexes (see \Cref{thm:BettiCM}). This is the content of \Cref{sect:FirstBound}.
	\item The second approach, presented is \Cref{sect:SecondBound}, employs the theory of \emph{lex-plus-squares} ideals to bound the Betti numbers of ideals containing many generators in degree $2$, including the squares of the variables. Again the result on balanced complexes given in \Cref{thm:BoundsBettiLexPlusSquares} follows as an immediate application.
	\item In \Cref{section:pseudo}, we focus on balanced normal pseudomanifolds. We use a result by Fogelsanger \cite{Fogelsanger} to derive upper bounds for the graded Betti numbers in the first strand of the graded minimal free resolution in this setting (see \Cref{thm:pseudomanifold}).
	\item In \cite{KN} \emph{cross-polytopal stacked} spheres were introduced as the balanced  analog of stacked spheres, in the sense that they minimize the face vector among all balanced spheres with a given number of vertices. In \Cref{sectionstacked} (\Cref{thm:BettiCross}) we compute the graded Betti numbers of those spheres, and show that they only depend on the number of vertices and on the dimension. The same behavior in known to occur for stacked spheres \cite{TH}. Moreover, we conjecture that the graded Betti numbers in the linear strand of their resolution provide upper bounds for the ones of any balanced normal pseudomanifold.
\end{itemize} 
As a service to the reader, in particular to help him compare the different bounds, we use the same example to illustrate the predicted upper bounds: Namely, the toy example is a $3$-dimensional balanced simplicial complex on $12$ vertices with each color class being of cardinality $3$. All computations and experiments have been carried out with the help of the computer algebra system Macaulay2 \cite{M2}.

\section{Preliminaries} \label{preliminaries}
\subsection{Algebraic background}
Let $S=\FF[x_1,\dots,x_n]$ denote the polynomial ring in $n$ variables over an arbitrary field $\FF$ and let $\mathfrak{m}$ be  its maximal homogeneous ideal, i.e., $\mathfrak{m}=(x_1,\ldots,x_n)$. Denote with $\Mon_i(S)$ the set of monomials of degree $i$ in $S$, and for $u\in\Mon_i(S)$ and a term order $<$, we let $\Mon_i(S)_{<u}$ be the set of monomials of degree $i$ that are smaller than $u$ with respect to $<$. For a graded $S$-module $R$ we use $R_i$ to denote its graded component of degree $i$ (including $0$), where we use the standard $\NN$-grading of $S$. The \emph{Hilbert function} of a quotient $S/I$, where $I\subseteq S$ is a homogeneous ideal is the function from $\NN\to \NN$ that maps $i$ to $ \dim_{\mathbb{F}}(R_i)$. A finer invariant can be obtained from the minimal graded free resolution of $S/I$. The \emph{graded Betti number} $\beta_{i,i+j}^S(S/I)$ is the non-negative integer 
	$$\beta^{S}_{i,i+j}(R):=\dim_{\mathbb{F}} \Tor_{i}^{S}(R,\mathbb{F})_{i+j}.$$ 
We often omit the superscript $S$, when the coefficient ring is clear from the context. 
 We refer to any commutative algebra book (e.g., \cite{BH-book}) for further properties of the graded minimal free resolution of $S/I$. 
\begin{definition}
Let $I\subseteq S$ be a homogeneous ideal and let $R=S/I$ be of Krull dimension $d$. 
	Let $\Theta=\{\theta_1,\dots,\theta_d\}\subseteq S_1$. Then:
	\begin{itemize}
		\item[(i)] $\Theta$ is a \emph{linear system of parameters} (\emph{l.s.o.p.} for short) for $R$ if $\dim(R/(\theta_1,\dots,\theta_i)R)=\dim(R)-i$, for all $1\leq i\leq d$.
		\item[(ii)] $\Theta$ is a \emph{regular sequence} for $R$ if $\theta_{i}$ is not a zero divisor of $\dim(R/(\theta_1,\dots,\theta_{i-1})R)$, for all $1\leq i\leq d$.
	\end{itemize}
\end{definition} 
We remark that due to the Noether normalization lemma, an l.s.o.p. for $R=S/I$ always exists, if $\FF$ is an infinite field. Moreover, if $\Theta$ is a regular sequence, then $\Theta$ is an l.s.o.p., but the converse is far from being true in general. The class of rings for which the converse holds is of particular interest.
\begin{definition}
	A graded ring $R$ is \emph{Cohen-Macaulay} if every l.s.o.p. is a regular sequence for $R$.
\end{definition}
The theory of Cohen-Macaulay rings plays a key role in combinatorial commutative algebra and its importance cannot be overstated (see e.g., \cite{BH-book,Stanley-greenBook}).

The next two statements will be useful for providing upper bounds for graded Betti numbers.
\begin{lemma}\label{technicalbetti}
	Let $R=S/I$ with $I$ an homogeneous ideal and $\theta\in S_1$. 
	\begin{itemize}
		\item[(i)] \cite[Corollary 8.5]{MN} If the multiplication map $\times\theta:\; R_k\longrightarrow R_{k+1}$ is injective for every $k\leq j$, then
		$$\beta^{S}_{i,i+k}(R)\leq \beta^{S/\theta S}_{i,i+k}(R/\theta R),$$
		for every $i\geq 0$ and $k\leq j$.
		\item[(ii)] \cite[Proposition 1.1.5]{BH-book} If $\theta$ is not a zero divisor of $M$, then 
		$$\beta^{S}_{i,i+j}(R)= \beta^{S/\theta S}_{i,i+j}(R/\theta R),$$
		for every $i,j\geq 0$.
	\end{itemize}
\end{lemma}
From \Cref{technicalbetti} it immediately follows that modding out by a regular sequence does not affect the graded Betti numbers.

\subsection{Lex ideals}
In order to show upper bounds for the graded Betti numbers we will make use of lexicographic ideals. As above, we let  $S=\FF[x_1,\dots,x_n]$. Given a monomial ideal $I\subseteq S$ we denote by $G(I)$ its unique set of minimal monomial generators and we use $G(I)_j$ to denote those monomials in $G(I)$ of degree $j$. Let $>_{\lex}$ be the \emph{lexicographic order} on $S$ with $x_1>_{\lex}\dots,>_{\lex}x_n$. I.e., we have $x_1^{a_1}x_2^{a_2}\cdots x_n^{a_n}>_{\lex}x_1^{b_1}x_2^{b_2}\cdots x_n^{b_n}$ if the leftmost non-zero entry of $(a_1-b_1,\ldots,a_n-b_n)$ is positive. A monomial ideal $L\subseteq S$ is called a \emph{lexicographic ideal} (or \emph{lex ideal} for short) if for any monomials $u\in L$ and $v\in S$ of the same degree, with $v>_{\lex} u$ it follows that $v\in L$. Macaulay \cite{Macaulay} showed that for any graded homogeneous ideal $I\subseteq S$ there exists a unique lex ideal, denoted with $I^{\lex}$, such that $S/I$ and $S/I^{\lex}$ have the same Hilbert function. In particular, the $\mathbb{F}$-vector space $I^{\lex}\cap S_i$ is spanned by the first $\dim_{\mathbb{F}}S_i-\dim_{\mathbb{F}}(S/I)_i$ largest monomials of degree $i$ in $S$. Note that the correspondence between $I$ and $I^{\lex}$ is far from being one to one, since $I^{\lex}$ only depends on the Hilbert function of $I$. We conclude this section with two fundamental results on the graded Betti numbers of lex ideals.
\begin{lemma}[Bigatti,\cite{Bigatti}, Hulett,\cite{Hulett}, Pardue \cite{Pardue}]\label{lemma:Bigatti}
	For any homogeneous ideal $I\subseteq S$ it holds that
	$$\beta_{i,i+j}^S(S/I)\leq\beta_{i,i+j}^{S}(S/I^{\lex}),$$
	for all $i,j\geq 0$.
\end{lemma}
\Cref{lemma:Bigatti} states that among all graded rings with the same Hilbert functions, the quotient with respect to the lex ideal maximizes all graded Betti number simultaneously. Another peculiar property of lex ideals is that their graded Betti numbers are determined just by the combinatorics of their minimal generating set $G(I^{\lex})$.  
For a monomial $u\in S$ denote with $\max(u)=\max\left\lbrace i~:~ x_i|u \right\rbrace$. 
\begin{lemma}[Eliahou-Kervaire,\cite{EK}\label{lemma:eliker}]
	Let $I^{\lex}\subseteq S$ be a lexicographic ideal. Then
	$$\beta_{i,i+j}^S(S/I^{\lex})=\displaystyle\sum_{u\in G(I^{\lex})\cap S_{j+1}}\binom{\max(u)-1}{i-1},$$
	for all $i\geq 1$, $j\geq 0$.
\end{lemma}
\subsection{Simplicial complexes}
An (abstract) \emph{simplicial complex} $\Delta$ on a (finite) vertex set $V(\Delta)$ is any collection of subsets of $V(\Delta)$ closed under inclusion. The elements of $\Delta$ are called \emph{faces}, and a face that is maximal with respect to inclusion is called \emph{facet}. The \emph{dimension} of a face $F$ is the number $\dim(F):=\left| F\right|-1$, and the dimension of $\Delta$ is $\dim(\Delta):=\max\left\lbrace\dim(F)~:~ F\in\Delta \right\rbrace$. In particular $\dim(\emptyset)=-1$. If all facets of $\Delta$ have the same dimension, $\Delta$ is said to be \emph{pure}. One of the most natural combinatorial invariants of a $(d-1)$-dimensional simplicial complex to consider is its \emph{f-vector} $f(\Delta)=(f_{-1}(\Delta),f_0(\Delta),\dots,f_{d-1}(\Delta))$, defined by $f_i(\Delta):=\left|\left\lbrace F\in\Delta~:~\dim(F)=i \right\rbrace \right|$ for $-1\leq i\leq d-1$. However, for algebraic and combinatorial reasons it is often more convenient to consider a specific  invertible linear transformation of $f(\Delta)$; namely 
$$h_j(\Delta)=\sum_{i=0}^{j}(-1)^{j-i}\binom{d-i}{d-j}f_{i-1}(\Delta),$$
for $0\leq i\leq d$. The vector $h(\Delta)=(h_0(\Delta),h_1(\Delta),\dots,h_d(\Delta))$ is called the \emph{h-vector} of $\Delta$.  

Given a subset $W\subseteq V(\Delta)$ we define the subcomplex
$$\Delta_W:=\left\lbrace F\in\Delta: F\subseteq W \right\rbrace,$$
and we call a subcomplex \emph{induced} if it is of this form. Another subcomplex associated to $\Delta$ is its \emph{$j$-skeleton} $$\Skel_{j}(\Delta):=\left\lbrace F\in \Delta~:~ \dim(F)\leq j \right\rbrace,$$
consisting of all faces of dimension at most $j$ (for $0\leq j\leq d-1$). For two simplicial complexes $\Delta$ and $\Gamma$, with $\dim(\Delta)=d-1$ and $\dim(\Gamma)=e-1$ we define the \emph{join} of $\Delta$ and $\Gamma$ to be the $(d+e-1)$-dimensional complex defined by
$$\Delta*\Gamma=\left\lbrace F\cup G~:~ F\in\Delta, G\in\Gamma \right\rbrace.$$    
The \emph{link} $\lk_{\Delta}(F)$ of a face $F\in \Delta$ describes $\Delta$ locally around $F$:
$$
\lk_\Delta(F):=\{G\in \Delta~:~G\cup F\in \Delta,\; G\cap F=\emptyset\}.
$$
%Simplicial complexes can be though as discretization of topological spaces: it is always possible to embed a simplicial complex $\Delta$ in some $\mathbb{R}^n$ by realizing its faces as geometric simplices. Since all the realizations are homeomorphic we call the \emph{geometric realization} of $\Delta$ this well defined topological space. We will not distinguish, with notations, between the abstract and topological objects: for example when we write the $\Delta$ is a sphere, it should be clear that we are referring to its geometric realization.\\
Simplicial complexes are in one to one correspondence to squarefree monomial ideals: Given a simplicial complex $\Delta$ with $V(\Delta)=[n]:=\{1,2,\ldots,n\}$ its \emph{Stanley-Reisner ideal} is the squarefree monomial ideal $I_{\Delta}\subseteq S$ defined by
$$I_{\Delta}:=\left( x_F~:~ F\notin\Delta \right)\subseteq S:=\FF[x_1,\ldots,x_n],$$
where $x_F=\prod_{i\in F}x_i$. The quotient $\mathbb{F}[\Delta]:=S/I_{\Delta}$ is called the \emph{Stanley-Reisner ring} of $\Delta$. It is well-known that $\dim(\mathbb{F}[\Delta])=\dim(\Delta)+1$.\\
This correspondence is extremely useful to study how algebraic invariants of the Stanley-Reisner rings reflect combinatorial and topological properties of the corresponding simplicial complex, and vice versa. A special instance for this is provided by Hochster's formula (see \cite[Theorem 5.5.1]{BH-book}):
\begin{lemma}[Hochster's formula]\label{Hochster}
	$$\beta_{i,i+j}(\mathbb{F}[\Delta])=\displaystyle\sum_{\substack{W\subseteq V(\Delta)\\\left| W\right| =i+j}} \dim_{\mathbb{F}}\widetilde{H}_{j-1}(\Delta_W;\mathbb{F}).$$
\end{lemma}
A simplicial complex $\Delta$ is called \emph{Cohen-Macaulay} over $\FF$ if $\mathbb{F}[\Delta]$ is a Cohen-Macaulay ring. As Cohen-Macaulayness (over a fixed field $\FF$) only depends on the geometric realization of $\Delta$,  Cohen-Macaulayness is a topological property (see e.g., \cite{Munkres}). In particular, triangulations of spheres and balls are Cohen-Macaulay over any field. Another crucial property of Cohen-Macaulay complexes is the following (see e.g., \cite{Stanley-greenBook}).
\begin{lemma}\label{hiCM}
	Let $\Delta$ be a $(d-1)$-dimensional Cohen-Macaulay simplicial complex and let $\Theta=(\theta_1,\dots,\theta_d)$ be an l.s.o.p.  for $\mathbb{F}[\Delta]$. Then
	$$h_i(\Delta)=\dim_{\mathbb{F}}\left( \mathbb{F}[\Delta]/\Theta\mathbb{F}[\Delta]\right)_i.$$ 
\end{lemma}

In \Cref{section:pseudo} we will be interested in another class of simplicial complexes, so-called \emph{normal pseudomanifolds}. We call a connected pure $(d-1)$-dimensional simplicial complex $\Delta$ a \emph{normal pseudomanifold} if every $(d-2)$-face of $\Delta$ is contained in exactly two facets and if the link of every face of $\Delta$ of dimension $\leq d-3$ is connected. 

We finally provide the definition of balanced simplicial complexes. 
\begin{definition}\label{defbalanced}
	A $(d-1)$-dimensional simplicial complex $\Delta$ is \emph{balanced} if there is a partition of its vertex set $V(\Delta)=\displaystyle\bigcupdot_{i=1}^d V_i$ such that $\left| F\cap V_i\right|\leq 1$, for every $i=1,\dots,d$. 
\end{definition} 
We often refer to the sets $V_i$ as \emph{color classes}. Another way to phrase this definition is to observe that $\Delta$ is balanced if and only if its $1$-skeleton is $d$-colorable, in the classical graph theoretic sense. Note that, without extra assumptions on its structure, a balanced simplicial complex does not uniquely determine the partition in color classes, nor their sizes, as shown by the middle and right complex in \Cref{3complexes}. However, in this article, we will always assume the vertex partition to be part of the data contained in $\Delta$. 
\begin{figure}[h]
	\centering
	\includegraphics[scale=1]{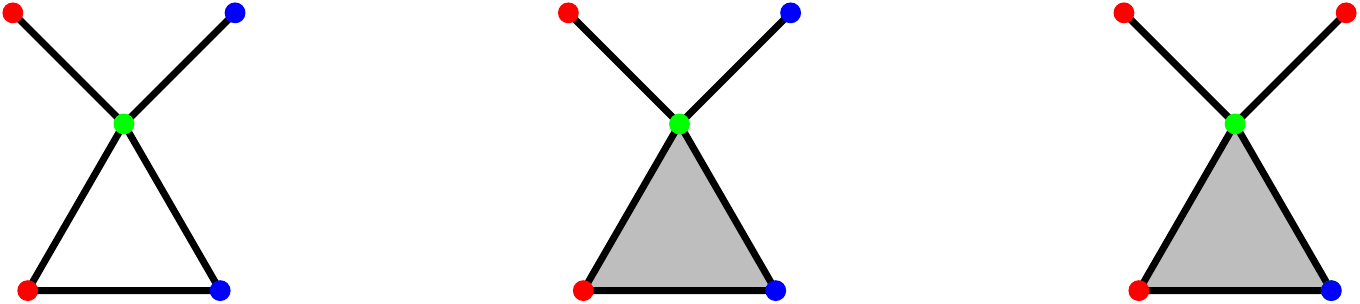}
	\caption{From left to right: a simplicial complex that is not balanced. Two balanced complexes with different partitions in color classes.}
	\label{3complexes}
\end{figure}     
The class of pure balanced simplicial complexes agrees with the class of so-called \emph{completely balanced} complexes, originally introduced by Stanley in \cite{St79}. However, a balanced simplicial complex in the sense of \Cref{defbalanced} does not need to be pure. 
We want to point out that a balanced simplicial complex cannot have \emph{too many} edges, since all monochromatic edges are forbidden. This idea will be made more precise and used intensively in the following sections.

\section{General balanced simplicial complexes}
\label{balsimp}
In the following, we consider arbitrary balanced simplicial complexes without assuming any further algebraic or combinatorial properties. 
Our aim is to prove explicit upper bounds for the graded Betti numbers of the Stanley-Reisner rings of those simplicial complexes. This will be achieved by exhibiting (non-balanced) simplicial complexes (one for each strand in the linear resolution), whose graded Betti numbers are larger than those of all balanced complexes on a fixed vertex partition.

We first need to introduce some notation. 
Recall that the \emph{clique complex} of a graph $G=(V,E)$ on vertex set $V$ and edge set $E$ is the simplicial complex $\Delta(G)$ on vertex set $V$, whose faces correspond to cliques of $G$, i.e.,
$$
\Delta(G):=\{F\subseteq V~:~ \{i,j\}\in E \mbox{ for all } \{i,j\}\subseteq F \mbox{ with }i\neq j\}.
$$
Let $\Delta$ be a $(d-1)$-dimensional balanced simplicial complex with vertex partition $V(\Delta)=\bigcupdot_{i=1}^d V_i$. Let $n_i:=|V_i|$ denote the sizes of the color classes of $V(\Delta)$. Throughout this section, we denote with $K_{n_1,\ldots,n_d}$ the complete $d$-partite graph on vertex set $\bigcupdot_{i=1}^d V_i$.  
Note that the $1$-skeleton of $\Delta$, considered as a graph, is clearly a subgraph of $K_{n_1,\ldots,n_d}$ and that, by definition of a clique complex, we have $\Delta\subseteq\Delta(K_{n_1,\dots,n_d})$.  

We can now state our first bound, though not yet in an explicit form.

\begin{theorem}\label{skeldpartite}
	Let $\Delta$ be a $(d-1)$-dimensional balanced simplicial complex on $V=\displaystyle\bigcupdot_{i=1}^d V_i$  with $n_i:=\left| V_i\right|$. 
	Then
	$$\beta_{i,i+j}\left( \mathbb{F}\left[\Delta\right] \right)\leq 
	\beta_{i,i+j}\left( \mathbb{F}\left[ \Skel_{j-1}\left(\Delta(K_{n_1,\ldots,n_d})\right)\right] \right)$$ 
	for every $i,j\geq 0$.
\end{theorem}
\begin{proof}
The proof relies on Hochster's formula. 
We fix $j\geq 0$. To simplify notation we set  $\Sigma= \Skel_{j-1}(\Delta(K_{n_1,\ldots,n_d}))$. Given a simplicial complex $\Gamma$, we denote by  $(C_{\bullet}(\Gamma),\partial_{j}^\Gamma)$ the chain complex which computes its  simplicial homology over $\mathbb{F}$.\\ 
Let $W\subseteq V$. 
	As $\dim \Sigma= j-1$, we have $\dim\left( \Sigma_W\right)\leq j-1$ and hence $C_{j}\left( \Sigma_W\right)=0$. This implies 
	\begin{equation}
	\label{eq:homology}
	\widetilde{H}_{j-1}\left( \Sigma_W;\mathbb{F}\right)= \ker \partial_{j-1}^{\Sigma_W}. 
		\end{equation}
	As $\Delta(K_{n_1,\ldots,n_d})$ is the ``maximal'' balanced simplicial complex with vertex partition $\bigcupdot_{i=1}^d V_i$, it follows that   $\Skel_{j-1}(\Delta)\subseteq \Sigma$ and thus $C_{j-1}\left( \Delta_W\right)\subseteq C_{j-1}\left( \Sigma_W\right)$. In particular, we conclude 
	$$
	\ker \partial_{j-1}^{\Delta_W} \subseteq \ker \partial_{j-1}^{\Sigma_W}
	$$
	and, using \eqref{eq:homology}, we obtain
	$$\dim_{\mathbb{F}}\widetilde{H}_{j-1}\left( \Delta_W;\mathbb{F}\right)\leq \dim_{\mathbb{F}}\widetilde{H}_{j-1}\left( \Sigma_W;\mathbb{F}\right).$$ 
	The claim follows from Hochster's formula (\Cref{Hochster}).
\end{proof}

We now provide a specific example of the bounds in \Cref{skeldpartite}.

\begin{example}
The graded Betti numbers of any $3$-dimensional balanced simplicial complex on $12$ vertices with $3$ vertices in each color class can be bounded by the graded Betti numbers of the skeleta of $\Gamma:=\Delta(K_{3,3,3,3})$. More precisely, we can bound $\beta_{i,i+j}(\FF[\Delta])$ by the corresponding Betti number of the $(j-1)$-skeleton of $\Gamma$. We record those numbers in the following table:
\begin{table}[h!]

	\[\begin{array}{r|l|l|l|l|l|l|l|l|l|l|l|c}
	\setminus  i & 0 &1 & 2 & 3 & 4 & 5 & 6 & 7 & 8 &9 &10 &11  \\ \hline
	\beta_{i,i+1}(\mathbb{F}[\Skel_{0}(\Gamma)])&0&66 & 440& 1485& 3168& 4620& 4752& 3465& 1760& 594& 120& 11\\
	\beta_{i,i+2}(\mathbb{F}[\Skel_{1}(\Gamma)])&0 &108& 945& 3312& 6720& 8856& 7875& 4720& 1836& 420& 43&0\\
	\beta_{i,i+3}(\mathbb{F}[\Skel_{2}(\Gamma)])&0 &81 & 648& 2376& 4752& 5733& 4352& 2052& 552& 65& 0&0\\
	\beta_{i,i+4}(\mathbb{F}[\Gamma])&0&0&0&0&81 & 216& 216& 96& 16&0&0&0\\
	\end{array}
	\]
	\caption{Graded Betti numbers of the skeleta of $\Gamma=\Delta(K_{3,3,3,3})$.}
	\label{tab:tableSkel}
\end{table}
\end{example}

\begin{remark}
Observe that the $(j-1)$-skeleton of the clique complex $\Delta(K_{n_1,\ldots,n_d})$ is  balanced if and only if $j=d$ (or, less interestingly, if $j=1$). It follows that the upper bounds for the graded Betti numbers of a $(d-1)$-dimensional balanced simplicial complex, given in  \Cref{skeldpartite}, are attained for the $d$\textsuperscript{th} (and trivially, the $0$\textsuperscript{th}) row of the Betti table. However, they are not necessarily sharp for the other rows of the Betti table and we do not expect them to be so.
\end{remark}

In order to turn the upper bounds from  \Cref{skeldpartite} into explicit ones, we devote the rest of this section to the computation of the graded Betti numbers of the skeleta of $\Delta(K_{n_1,\dots,n_d})$. 
We first consider $\Delta(K_{n_1,\dots,n_d})$. As a preparation we determine the homology of induced subcomplexes of $\Delta(K_{n_1,\dots,n_d})$.

\begin{lemma}\label{homology_dpartite}
Let $\Gamma=\Delta(K_{n_1,\dots,n_d})$ with vertex partition $V:=\displaystyle\bigcupdot_{i=1}^d V_i$. For $W\subseteq V$, set $W_i:=W\cap V_i$, for $1\leq i\leq d$ and $\{ i_1,\ldots, i_k\}:=\{i~:~ W_i\neq \emptyset\}$. Then
	$$\widetilde{H}_{j-1}\left(\Gamma_W; \mathbb{F}\right)=
	\begin{cases}
	\mathbb{F}^{\left| W_{i_1}\right|-1 }\otimes_{\mathbb{F}}\dots\otimes_{\mathbb{F}}\mathbb{F}^{\left| W_{i_k}\right|-1 }, & \mbox{if }k=j\\
		0, &\mbox{if } k\neq j
	\end{cases}.$$
	In particular, $\widetilde{H}_{j-1}\left(\Gamma_W; \mathbb{F}\right)\neq 0$ if and only if $k=j$ and $|W_{i_\ell}|\geq 2$ for $1\leq \ell\leq k$.
\end{lemma}
\begin{proof}
Denoting by $\overline{V_i}$ the simplicial complex consisting of the isolated vertices in $V_i$, we can write $\Gamma$ as the join of those $\overline{V_i}$:
	\begin{equation}\label{eq:gamma is a join}
	\Gamma=\overline{V_1}*\dots*\overline{V_d}. 
	\end{equation}
In particular, we have
	$$\Gamma_W=\overline{W_{i_1}}*\dots*\overline{W_{i_k}}.$$
Using the K\"unneth formula for the homology of a join (see e.g., \cite[\S 58]{Munkres}) and the fact that $$\widetilde{H}_j\left( \overline{W_i};\mathbb{F}\right)=\begin{cases}
	\mathbb{F}^{\left| W_i\right|-1 }, &\mbox{if } j =0\\
		0 &\mbox{if } j\neq 0,
	\end{cases}$$
	we deduce the desired formula for the homology. The ``In particular''-part follows directly from this formula.
\end{proof}	

\begin{remark}\label{GammaCM}
 Since Cohen-Macaulayness is preserved under taking joins and since every $0$-dimensional simplicial complex is Cohen-Macaulay, it follows directly from \eqref{eq:gamma is a join} that the clique complex $\Delta(K_{n_1,\dots,n_d})$ is a Cohen-Macaulay complex. Accordingly, the same is true for the skeleta of $\Delta(K_{n_1,\ldots,n_d})$. 
%since Cohen-M it holds that  $\lk_{\Gamma}(F)=\Gamma_{V(\Gamma)\setminus (V_{i_1}\cup\dots,\cup V_{i_k})}$for any $(k-1)$-dimensional face $F\in\Gamma$ intersecting $V_{i_1},\dots,V_{i_k}$. In particular this shows that the links of faces in $\Gamma$ are induced subcomplexes, and by \Cref{homology_dpartite} these complexes can have non trivial homology only in the top degree. Via Reisner's theorem (see \cite[Corollary 4.2]{Stanley-greenBook}) it follows that $\Gamma$ is Cohen-Macaulay. Since the skeleta of a Cohen-Macaulay complex are Cohen-Macaulay, so is $\Skel_{j-1}(\Gamma)$.
\end{remark}
\Cref{homology_dpartite} enables us to compute the graded Betti numbers of $\Delta(K_{n_1,\ldots,n_d})$.
\begin{lemma}\label{ugly}
Let $d,n_1,\ldots,n_d$ be positive integers. Then
	\begin{equation}
	\label{eq:Betti ugly}
	\beta_{i,i+j}\left( \mathbb{F}\left[ \Delta(K_{n_1,\dots,n_d})\right] \right)=\displaystyle\sum_{\substack{I\subseteq[d]\\ I=\left\lbrace i_1,\dots, i_j \right\rbrace  }}\left( \displaystyle\sum_{\substack{c_1+\dots +c_j=i \\c_\ell\geq 1, \forall \ell\in\left[ 1,j\right]}}\left( \prod_{\ell=1}^j c_\ell\cdot\binom{n_{i_\ell}}{c_\ell-1} \right)\right)  
		\end{equation}
	for $i,j\geq 0$.
In particular, if $n_1=\cdots=n_d=k$, then 
	$$\beta_{i,i+j}\left( \mathbb{F}\left[ \Delta(K_{k,\ldots, k})\right] \right)=\binom{d}{j}\left( \displaystyle\sum_{\substack{c_1+\dots +c_j=i \\c_\ell\geq 1, \forall \ell\in\left[ 1,j\right]}}\left( \prod_{\ell=1}^j c_\ell\cdot\binom{k}{c_\ell-1} \right)\right) $$
	for $i,j\geq 0$.

\end{lemma}
\begin{proof}
We prove the statement by a direct application of Hochster's formula. Fix $i,j\geq 0$. By \Cref{homology_dpartite} and \Cref{Hochster}, to compute $\beta_{i,i+j}(\Delta(K_{n_1,\ldots,n_d}))$, we need to count subsets $W\subseteq \bigcupdot_{\ell=1}^d V_i$ such that $|\{\ell~:~W\cap V_\ell\neq \emptyset\}|=j$ and $|W\cap V_\ell|\neq 1$ for $1\leq \ell\leq d$. To construct such a set, we proceed as follows: 
\begin{itemize}
\item  We first choose $i_1<\cdots <i_j$ such that $W\cap V_{i_\ell}\neq \emptyset $ for $1\leq \ell \leq j$. 
\item Next, for  each $i_\ell$ we pick an integer $c_t\geq 2$, with the property that $c_1+\cdots +c_j=i+j$.
		\item Finally, there are $\binom{n_{i_\ell}}{c_\ell}$ ways to choose $c_\ell$ vertices among the $n_{i_\ell}$ vertices of $V_{i_\ell}$.
	\end{itemize}
	By \Cref{homology_dpartite} the dimension of the $(j-1)$\textsuperscript{st} homology of such a subset $W$ equals  $\prod_{\ell=1}^j\left( c_\ell-1\right)$. Combining the previous argument, we deduce the required formula \Cref{eq:Betti ugly}.
	The second statement now is immediate.
\end{proof}	

We illustrate \Cref{eq:Betti ugly} with an example.
\begin{example}
	Consider the clique complex $\Delta(K_{3,3,2})$ of $K_{3,3,2}$. To compute $\beta_{3,5}(\mathbb{F}[\Delta(K_{3,3,2})])$, we need to consider the $2$-element subsets of $[3]$.
	
	For the set $\{1,2\}$ the inner sum in \eqref{eq:Betti ugly} equals
	\begin{equation*}
	\sum_{\substack{c_1+c_2=3\\c_1,c_2\geq 1}}c_1\cdot c_2\cdot \binom{3}{c_1-1}\cdot\binom{3}{c_2-1}=12,
	\end{equation*}
	since the sum has two summands (corresponding to $(c_1,c_2)\in \{(1,2),(2,1)\}$), each contributing with $6$.\\
	Similarly, for $\{1,3\}$ and $\{2,3\}$, we obtain $2$ for the value of the inner sum. In total, this yields:
	$$\beta_{3,5}(\mathbb{F}[\Delta])=12+2+2=16.$$
	\end{example}
	
	We now turn our attention to the computation of the graded Betti numbers of the skeleta of $\Delta(K_{n_1,\dots,n_d})$. The following result, which is a special case of  \cite[Theorem 3.1]{2015arXiv150205670R} by Roksvold and Verdure, is crucial for this aim.
\begin{lemma}\label{roksvold_verdure}
		Let $\Delta$ be a $(d-1)$-dimensional Cohen-Macaulay complex with $f_0(\Delta)=n$. Then
		$$\beta_{i,i+j}\left( \mathbb{F}\left[ \Skel_{d-2}(\Delta)\right] \right)=
		\begin{cases}
		\beta_{i,i+j}\left( \mathbb{F}\left[ \Delta\right] \right), & \mbox{ if }j<d-1\\
		\beta_{i,i+d-1}\left( \mathbb{F}\left[ \Delta\right] \right)-\beta_{i-1,i+d-1}\left( \mathbb{F}\left[ \Delta\right] \right)+\binom{n-d}{i-1}f_{d-1}\left( \Delta\right), & \mbox{ if }j=d-1\\
		0, &\mbox{ if } j\geq d
		\end{cases}$$
		for $0\leq i\leq n-d+1$. 
\end{lemma}
	Applying \Cref{roksvold_verdure} iteratively, we obtain the following recursive formula for the graded Betti numbers of general skeleta of a Cohen-Macaulay complex:
\begin{corollary}\label{cor:skeletaBetti}
Let $s$ be a positive integer and let $\Delta$ be a $(d-1)$-dimensional Cohen-Macaulay complex with $f_0(\Delta)=n$. Set  $\Sigma=\Skel_{d-s-1}(\Delta)$. Then
	$$\beta_{i,i+j}\left( \mathbb{F}\left[ \Sigma\right] \right)=
	\begin{cases}
	\beta_{i,i+j}\left( \mathbb{F}\left[ \Delta\right] \right), &\mbox{ if } j<d-s\\
	\sum_{k=0}^s(-1)^k\beta_{i-k,i+d-s}\left( \mathbb{F}\left[ \Delta\right] \right)+\sum_{t=0}^{s-1}(-1)^{t-s+1}\binom{n-d+t}{i-s+t}f_{d-t-1}\left( \Delta\right),  & \mbox{ if }j=d-s\\
	0, & \mbox{ if }j\geq d-s+1
	\end{cases}$$
	for $0\leq i\leq n-d+s$.
\end{corollary}

Since the clique complex $\Delta(K_{n_1,\ldots,n_d})$ is Cohen-Macaulay (see \Cref{GammaCM}), we can use \Cref{cor:skeletaBetti} to compute the graded Betti numbers of its skeleta. Combining this with \Cref{skeldpartite}, we obtain the following bounds for the graded Betti numbers of an arbitrary balanced simplicial complex.

\begin{corollary}
	Let $\Delta$ be a $(d-1)$-dimensional balanced simplicial complex on vertex set $V=\displaystyle\bigcupdot_{i=1}^d V_i$, with $n:=|V|$ and $n_i:=\left| V_i\right|$. Let $\Gamma=\Delta(K_{n_1,\dots,n_d})$. Then
	$$\beta_{i,i+j}\left( \mathbb{F}\left[ \Delta \right] \right)\leq \sum_{k=0}^{d-j}(-1)^k\beta_{i-k,i+j}\left( \mathbb{F}\left[ \Gamma\right] \right)+\sum_{t=0}^{d-j-1}(-1)^{t-d+j+1}\binom{n-d+t}{i-d+j+t}f_{d-t-1}\left( \Gamma\right).$$
\end{corollary}

Note that the graded Betti numbers of $\Gamma:=\Delta(K_{n_1,\ldots,n_d})$ are given in \Cref{ugly} and that the $f$-vector of $\Gamma$ is given by
$$
f_i(\Gamma)=\sum_{I\subseteq [d],|I|=i+1}\prod_{\ell\in I}n_\ell
$$
for $0\leq i\leq d-1$. Therefore, the previous corollary really provides explicit bounds for the graded Betti numbers of a balanced simplicial complex.

\section{A first bound in the Cohen-Macaulay case}\label{sect:FirstBound}
We let $S=\FF[x_1,\ldots,x_n]$ denote the polynomial ring in $n$ variables over an arbitrary field $\FF$. 
The ultimate aim of this section is to show upper bounds for the graded Betti numbers of the Stanley-Reisner rings of balanced Cohen-Macaulay complexes. On the way, more generally, we will prove upper bounds for the graded Betti numbers of Artinian quotients $S/I$, where $I \subseteq S$ is a homogeneous ideal having \emph{many} generators in degree $2$.

\subsection{Ideals with many generators in degree $2$}
Throughout this section, we let $I\subsetneq S$ be a homogeneous ideal that has no generators in degree $1$, i.e., $I\subseteq\mathfrak{m}^2$. 

First assume that $S/I$ is of dimension $0$. 
It is  well-known and essentially follows from \Cref{lemma:Bigatti} by passing to the lex ideal $I^{\lex}$, that we can bound $\beta_{i,i+j}(S/I)$ by the corresponding Betti number  $\beta_{i,i+j}(S/\mathfrak{m}^{j+1})$ of the quotient of $S$ with the $(j+1)$\textsuperscript{st} power of the maximal homogeneous ideal $\mathfrak{m}\subseteq S$. \Cref{lemma:eliker} then yields
$$
\beta_{i,i+j}(S/I)\leq \binom{i-1+j}{j}\binom{n+j}{i+j}
$$
for all $i\geq 1$, $j\geq0$. Moreover, if $S/I$ is Cohen-Macaulay of dimension $d$, then, by modding out a linear system of parameters $\Theta\subseteq S$ (which is a regular sequence by assumption) and using \Cref{technicalbetti}, we can reduce to the $0$-dimensional case, which yields the well-known upper bound (see e.g., \cite[Lemma 3.4 (i)]{MUR}):
$$
\beta_{i,i+j}(S/I)\leq 
\binom{i-1+j}{j}\binom{n-d+j}{i+j},$$
for all $i\geq 1$, $j\geq0$. In particular, those bounds apply to Stanley-Reisner rings of Cohen-Macaulay complexes.   Moreover, if equality holds in the $j$\textsuperscript{th} strand, then $I$ has $(j+1)$-linear resolution (see e.g., \cite{HH} for the precise definition).  

In the following, assume that $S/I$ is Artinian and that there exists a positive integer $b$ such that
$$
\dim_\FF(S/I)_2\leq \binom{n+1}{2}-b.
$$
In other words, $I$ has at least $b$ generators in degree $2$. Our goal is to prove upper bounds for $\beta_{i,i+j}(S/I)$ in this setting. This will be achieved using similar arguments as the ones we just recalled that are used in the general setting. First, we need some preparations.

As, by assumption, $I$ does not contain polynomials of degree $1$, neither does its lex ideal $I^{\lex}\subseteq S$. In particular, we have 
$$
|G(I^\lex)\cap S_2|\geq b
$$
and $I^\lex$ contains at least the $b$ \emph{largest} monomials of degree $2$  in lexicographic order. The next lemma describes this set of monomials explicitly. 

\begin{lemma}\label{p and q}
Let $n\in \NN$ be a positive integer and let $b<\binom{n+1}{2}$. Let $x_px_q$  be the $b$\textsuperscript{th} largest monomial in the lexicographic order of degree $2$ monomials in variables $x_1,\ldots,x_n$ and assume $p\leq q$. Then:
 $$p=n-\left\lfloor\dfrac{\sqrt{-8b+4n(n+1)+1}}{2}-\frac{1}{2}\right\rfloor,$$
	and
	$$q=b+\dfrac{(p-1)(p-2n)}{2}.$$
\end{lemma}	

Since the proof of this lemma is technical and since the precise statement is not used later, we defer its proof to the appendix.

Intuitively, if a lex ideal $J\subseteq S$ has \emph{many} generators in degree $2$, then there can only exist relatively few generators of higher degree. More precisely, the next lemma provides a necessary condition for a monomial $u$ to lie in $G(J)_j$ for $j>2$ and thus enables us to bound the number of generators of $J$ of degree $j$. 

\begin{lemma}\label{subset}
	Let $j>2$ be an integer and let $J\subseteq S$ be a lex ideal. Let $x_px_q$ be the lexicographically smallest monomial of degree $2$ that is contained in $J$. If $u\in G(J)_j$ is a minimal generator of $J$ of degree $j$, then $u<_\lex x_px_qx_{n}^{j-2}$. In other words:
	$$
	G(J)_j\subseteq \Mon_j(S)_{<_\lex x_px_qx_{n}^{j-2}}.
	$$
\end{lemma}	
\begin{proof}
To simplify notation we set $w=x_px_qx_n^{j-2}$. First note that any monomial of degree $j$ that is divisible by $x_px_q\in G(J)$ cannot be a minimal generator of $J$. Let $u$ be a monomial of degree $j$ with $u>_\lex w$, that is not divisible by $x_px_q$. Then, there exists $\ell<p$ such that $x_\ell$ divides $u$ or $u$ is divisible by $x_p$ and there exists $p\leq r<q-1$ such that $x_px_r$ divides $u$. In the first case, let $x_r$ such that $x_\ell x_r$ divides $u$. Then, $x_\ell x_r>_\lex x_px_q$ and hence $x_\ell x_r\in J$, since $J$ is a lex ideal. This implies $u\notin G(J)$.
Similarly, in the second case, we have $x_px_r>_\lex x_px_q\in G(J)$ and hence 
 $u\notin G(J)$. The claim follows.
\end{proof}

Recall that a homogeneous ideal $I\subseteq S$, which is generated in degree $d$, is called \emph{Gotzmann ideal} if the number of generators of $\mathfrak{m}I$ is smallest possible. More generally, a graded ideal $I\subseteq S$ is called \emph{Gotzmann ideal} if all components $I_{\langle j\rangle}$ are Gotzmann ideals. Here, $I_{\langle j\rangle}$ denotes the ideal generated by all the elements in $I$ of degree $j$. By Gotzmann's persistence theorem \cite{Gotzmann}, a graded ideal $I\subseteq S$ is Gotzmann if and only if $I$ and $(I^\lex)_{\langle d\rangle}$ have the same Hilbert function. Moreover, as shown in \cite[Corollary 1.4]{HerzogHibi:ComponentwiseLinear}, this is equivalent to $S/I$ and $S/I^\lex$ having the same graded Betti numbers,
i.e., 
\begin{equation}
\label{eq:Gotzmann}
\beta_{i,i+j}(S/I)=\beta_{i,i+j}(S/I^\lex)
\end{equation}
for all $i,j\geq 0$. 
We state an easy lemma, which well be helpful to prove the main result of this section.

\begin{lemma}\label{prop:Gotzmann}
Let $j\geq d$ be positive integers and let $J\subseteq S$ be a Gotzmann ideal that is generated in degree $d$. Let $I=J+\mathfrak{m}^{j+1}$. Then 
\begin{equation}\label{eq:equality}
\beta_{i,i+\ell}(S/I)=\beta_{i,i+\ell}(S/I^\lex)
\end{equation}
for all $i,\ell\geq 0$. 
\end{lemma}

\begin{proof}
We first note that, as $J$ is Gotzmann, so are its graded components $I_{\langle j\rangle}$. Moreover, as any power of $\mathfrak{m}$ is Gotzmann, it follows from the definition of a Gotzmann ideal that $I$ has to be Gotzmann as well. The claim now follows from \cite[Corollary 1.4]{HerzogHibi:ComponentwiseLinear}.
\end{proof}
We can now state the main result of this section.

\begin{theorem}\label{thm:BettiCM}
Let $I\subseteq S $ be a homogeneous ideal, that does not contain linear forms. Let $\dim_\FF (S/I)_2\leq\binom{n+1}{2}-b$ for some positive integer $b$. Let $x_px_q$, where $p\leq q$,  be the $b$\textsuperscript{th} largest monomial of degree $2$ in lexicographic order on $S$. Then
\begin{equation}\label{eq:BettiBoundCM}
\beta_{i,i+j}(S/I)\leq \displaystyle\sum_{\ell=p+1}^{n}\binom{\ell-p+j-1}{j}\binom{\ell-1}{i-1}+\displaystyle\sum_{\ell=q+1}^{n}\binom{\ell-q+j-2}{j-1}\binom{\ell-1}{i-1},
\end{equation}
	for any $i\geq 0$ and $j\geq 2$. Moreover if $I=J+\mathfrak{m}^{j+1}$, where $J\subseteq S$ is a Gotzmann ideal that is generated by $b$ elements of degree $2$, then equality is attained for a fixed $j\geq 2$ and all $i\geq 0$.
\end{theorem}

\begin{proof}
We fix $j\geq 2$ and we set $w:=x_px_qx_n^{j-1}$. 
By \Cref{lemma:Bigatti} we can use the graded Betti numbers of the lex ideal  $I^\lex\subseteq S$ of $I$ to bound the ones of $I$. Using \Cref{lemma:eliker} we infer 
	\begin{align*}
\beta_{i,i+j}(S/I)\leq \beta_{i,i+j}(S/I^{\lex})&\overset{\hspace{48pt}}{=}\displaystyle\sum_{u\in G(I^{\lex})_{j+1}}\binom{\max(u)-1}{i-1}\hspace{10pt                                                                       }\\
&	\overset{(\text{\Cref{subset}})}{\leq} \displaystyle\sum_{u\in \Mon_{j+1}(S)_{<w}}\binom{\max(u)-1}{i-1} \\
&	\overset{\hspace{48pt}}{=}\displaystyle\sum_{\substack{u\in \Mon_{j+1}(S)_{<w}\\
				x_p|u}}\binom{\max(u)-1}{i-1}+\displaystyle\sum_{\substack{u\in \Mon_{j+1}(S)_{<w}\\
				x_p\nmid u}}\binom{\max(u)-1}{i-1}.
%&=\displaystyle\sum_{\ell=p+1}^{n}\binom{\ell-p+j-1}{j}\binom{\ell-1}{i-1}+\displaystyle\sum_{\ell=q+1}^{n}\binom{\ell-q+j-2}{j-1}\binom{\ell-1}{i-1},
	\end{align*}
Let $u$ be a monomial of degree $j+1$, such that $u<_\lex w$. If $x_p|u$, then $\max(u)\geq q+1$ and $u$ is of the form $x_px_{\max(u)}\cdot v$, where $v$ is a monomial in  $\FF[x_{q+1},\ldots,x_{\max(u)}]$ of degree $j-1$. In particular, there are $\binom{(\ell-q)+(j-1)-1}{j-1}$ many such monomials with $\max(u)=\ell$.  Similarly, if $u$ is not divisible by $x_p$, then $\max(u)\geq p+1$ and $u$ is of the form $x_{\max(u)}\cdot v$, where $v$ is a monomial of degree $j$ in $\FF[x_{p+1},\ldots,x_{\max(u)}]$. There are $\binom{(\ell-p)+j-1}{j}$ many such monomials with $\max(u)=\ell$. The desired inequality follows.

For the equality case first note that if $I=J+\mathfrak{m}^{j+1}$, where $J$ is a Gotzmann ideal generated in degree $d$, then it follows from \Cref{prop:Gotzmann} that $\beta_{i,i+j}(S/I)=\beta_{i,i+j}(S/I^\lex)$ for all $i$. Moreover, as $I^\lex=\Lex(b)+\mathfrak{m}^{j+1}$, where $\Lex(b)$ denotes the lex ideal generated by the $b$ lexicographically largest monomials of degree $2$, the lex ideal $I^\lex$ attains equality in \Cref{eq:BettiBoundCM}.
%$I^{\lex}$ attains equality in \Cref{eq:BettiBoundCM}  for a fixed $j$ if and only if  $G(I^{\lex})\cap S_{j+1}=\Mon_{j+1}(S)_{<w}$, which is the case exactly if $G(I^{\lex})\cap S_\ell=\emptyset$ for $3\leq \ell\leq j$ and $(I^{\lex})_{\langle j+1\rangle}=\mathfrak{m}^{j+1}$. We can thus infer that $I^\lex=\Lex(b)+\mathfrak{m}^{j+1}$, where $\Lex(b)$ denotes the lex ideal generated by the $b$ lexicographically largest monomials of degree $2$. 

%Now assume that $I$ attains equality in \Cref{eq:BettiBoundCM}. Since $I$ and $I^{\lex}$ have the same Hilbert function and since the Hilbert function of $I^\lex$ exhibits the minimal possible growth \cite{Macaulay}, we infer from the previous discussion that $I_{\langle j+1\rangle}=\mathfrak{m}^{j+1}$ and $G(I)\cap S_\ell=\emptyset$ for $3\leq \ell\leq j$. It follows that $I=J+\mathfrak{m}^{j+1}$, where $J\subseteq S$ is a homogeneous ideal that is generated by $b$ degree $2$ elements. \Cref{prop:Gotzmann} implies that $J$ is a Gotzmann ideal. 

\end{proof}	

\begin{remark}
It is worth remarking that if an ideal $I$  attains equality in \eqref{eq:BettiBoundCM} for a fixed $j$, then the ideal $J$  (where $I=J+\mathfrak{m}^{j+1}$ as above) is not necessarily a monomial ideal.
E.g., for $n=2$ and $b=2$ the ideals
$$ 
(x_1^2,x_1x_2)+(x_1,x_2)^3 \quad \mbox{and}\quad (x_1^2+x_1x_2,x_2^2+x_1x_2)+(x_1,x_2)^3
$$
both maximize $\beta_{i,i+2}$ for any $i$. The maximal Betti numbers in this case are $\beta_{1,3}=\beta_{2,4}=1$.
\end{remark}

\subsection{Application: Balanced Cohen-Macaulay complexes}
The aim of this section is to use the results from the previous section in order to derive upper bounds for the graded Betti numbers of balanced Cohen-Macaulay complexes. 

In the following, let $\Delta$ be a balanced Cohen-Macaulay simplicial complex and let $\Theta\subseteq\FF[\Delta]$ be a linear system of parameters for $\FF[\Delta]$. In order to apply \Cref{thm:BettiCM} we need to bound the Hilbert function of the Artinian reduction $\FF[\Delta]/\Theta \FF[\Delta]$ in degree $2$ from above. As $\Delta$ is Cohen-Macaulay, it follows from \Cref{hiCM} that
$$
\dim_\FF\left(\FF[\Delta]/\Theta \FF[\Delta]\right)_2=h_2(\Delta),
$$
which implies that we need to find an upper bound  for $h_2(\Delta)$ or, equivalently, for the number of edges $f_1(\Delta)$.
  
\begin{lemma}\label{boundh2}
	Let $\Delta$ be a $(d-1)$-dimensional balanced simplicial complex with vertex partition $V(\Delta)=\displaystyle\bigcupdot_{i=1}^d V_i$. Let $n:=\left| V\right|$ and $n_i:=\left| V_i\right|$. Then
	\begin{equation}\label{eq:h2}
	h_2(\Delta)\leq \binom{n-d+1}{2}-\displaystyle\sum_{i=1}^{d}\binom{n_i}{2}.
	\end{equation}
\end{lemma}	 
\begin{proof}
	As $\Delta$ is balanced, it does not have monochromatic edges, i.e., we have  $\left\lbrace v,w\right\rbrace\notin\Delta$, if $v$ and $w$ belong to the same color class $V_i$ ($1\leq i\leq d$). As there are $\binom{n_i}{2}$ monochromatic non-edges of color $i$,  this gives the following upper bound for $f_1(\Delta)$:
	
	$$f_1(\Delta)\leq\binom{n}{2}-\displaystyle\sum_{i=1}^{d}\binom{n_i}{2}.$$
	  The claim now directly follows from the relation
	\begin{equation*}
	h_2(\Delta)=\binom{d}{2}-(d-1)f_0(\Delta)+f_1(\Delta).
	\end{equation*}  
\end{proof}	
%To simplify notation, we set 
%$$
%b(\Delta):=\displaystyle\sum_{i=1}^{d}\binom{n_i}{2},
%$$
%where $\Delta$ is as in \Cref{boundh2}. 

A direct application of \Cref{thm:BettiCM} combined with \Cref{boundh2} finally yields:

\begin{theorem}\label{bound_precise}
	Let $\Delta$ be a $(d-1)$-dimensional balanced Cohen-Macaulay complex with vertex partition $V=\displaystyle\bigcupdot_{i=1}^d V_i$. Let $n:=|V|$, $n_i:=\left| V_i\right|$ and $b:=\displaystyle\sum_{i=1}^d\binom{n_i}{2}$. Let $x_px_q$ be the $b$\textsuperscript{th} largest degree $2$ monomial of $\FF[x_1,\ldots,x_{n-d}]$ in lexicographic order with $p\leq q$. Then
	$$\beta_{i,i+j}(\mathbb{F}\left[ \Delta\right] )\leq \displaystyle\sum_{\ell=p+1}^{n-d}\binom{\ell-p+j-1}{j}\binom{\ell-1}{i-1}+\displaystyle\sum_{\ell=q+1}^{n-d}\binom{\ell-q+j-2}{j-1}\binom{\ell-1}{i-1},$$
	for any $i\geq 0$ and $2\leq j\leq d$.
\end{theorem}

The above statement is trivially true also for $j>d$. However, as the Castelnuovo-Mumford regularity of $\FF[\Delta]$ is at most $d$, we know that $\beta_{i,i+j}(\FF[\Delta])=0$ for any $i\geq 0$ and $j>d$. 

\begin{proof}
Let $S=\FF[x_1,\ldots,x_n]$. 
Let $\Theta$ be an l.s.o.p. for $\FF[\Delta]$. It follows from \Cref{technicalbetti} that 
$$\beta_{i,i+j}^S(\FF[\Delta])=\beta_{i,i+j}^{S/\Theta S}(S/(I_\Delta+(\Theta))).$$
Moreover, $S/ \Theta S\cong \FF[x_1,\ldots,x_{n-d}]:=R$ as rings and there exists a homogeneous ideal $J\subseteq R$ with  $\FF[\Delta]/\Theta\FF[\Delta]\cong R/J$ and $\beta_{i,i+j}^R(R/J)=\beta_{i,i+j}^{S/\Theta S}(S/(I_\Delta+(\Theta)))$. In particular, as $\Delta$ is Cohen-Macaulay, $\dim_\FF (R/J)_2=h_2(\Delta)$ satisfies the bound from \Cref{boundh2}. As $h_1(\Delta)=\dim_\FF (R/J)_1$, the ideal $J$ does not contain any linear form and the result now follows from \Cref{thm:BettiCM}.
\end{proof} 

\begin{remark}
Whereas we have seen that the bounds in \Cref{thm:BettiCM} are tight, the ones in \Cref{bound_precise} are not. For example, consider  the case that $n_1=n_2=2$ and $d=2$. In this situation, we have $b:=\sum_{i=1}^d \binom{n_i}{2}=2$ and $x_1x_2$ is the second largest degree $2$ monomial in the lexicographic order. \Cref{bound_precise} gives $\beta_{1,3}\leq 1$. However, by Hochster's formula, if $\Delta$ is a $1$-dimensional simplicial complex with $\beta_{1,3}(\FF[\Delta])=1$, then $\Delta$ must contain an induced $3$-cycle. But this means that $\Delta$ cannot be balanced. 
\end{remark}

\begin{example}\label{example:bound}
Let $\Delta$ be a $3$-dimensional balanced Cohen-Macaulay complex with $3$ vertices in each color class, i.e., $n_i=3$ for $1\leq i\leq 4$. We have $b:=\sum_{i=1}^4\binom{3}{2}=12$ and $x_2x_5$ is the $12$\textsuperscript{th} largest monomial of degree $2$ in variables $x_1,\ldots,x_8$. The bounds from \Cref{bound_precise} are recorded in the following table:
\begin{table}[H]
	\[\begin{array}{r|l|l|l|l|l|l|l|l|c}
	j\setminus i & 0 &1 & 2 & 3 & 4 & 5 & 6 & 7 & 8  \\ \hline
	2&0&62&360&915&1317&1156&617&185&24\\
	3&0&136&821&2155&3184&2855&1551&472&62\\\
	4&0&267&1653&4432&6665&6065&3336&1026&136\\
	\end{array}
	\]
		\label{tab:table1}
\end{table}
We set $S=\FF[x_1,\ldots,x_8]$ and we let $I\subseteq S$ be the lex ideal generated by the $12$ largest monomials of degree $2$ in variables $x_1,\ldots,x_8$. It follows from \Cref{thm:BettiCM} that $\beta_{i,i+j}(S/(I+\mathfrak{m}^{j+1}))$ equals the entry of the above table in the row, labeled $i$ and the column, labeled $j$. Moreover, it is shown in the proof of \Cref{thm:BettiCM} that $\beta_{i,i+\ell}(S/(I+\mathfrak{m}^{j+1}))=0$ if $\ell\notin \{1,j\}$. 
One can easily compute that for any $j$ the first row of the Betti table of $S/(I+\mathfrak{m}^{j+1})$ is given by
\begin{table}[H]
	\[\begin{array}{r|l|l|l|l|l|l|l|l|c}
	j\setminus i & 0 &1 & 2 & 3 & 4 & 5 & 6 & 7 & 8  \\ \hline
	1&0&12&38&66&75&57&28&8&1\\
	\end{array}
	\]
\end{table}
Finally, we compare the bounds from the upper table with the numbers $\beta_{i,i+j}(S/\mathfrak{m}^j)$, for general $3$-dimensional Cohen-Macaulay complexes on $12$ vertices. Those are displayed in the next table:
\begin{table}[H]
	\[\begin{array}{r|l|l|l|l|l|l|l|l|c}
	j\setminus i & 0 &1 & 2 & 3 & 4 & 5 & 6 & 7 & 8  \\ \hline
	2&0&120&630&1512&2100&1800&945&280&36\\
	3&0&330&1848&4620&6600&5775&3080&924&120\\\
	4&0&792&4620&11880&17325&15400&8316&2520&330\\
	\end{array}
	\]
	\end{table}
\end{example}

We point out that while \Cref{bound_precise} provides bounds for $\beta_{i,i+j}(\FF[\Delta])$ for all $i$ and all $j\geq 2$, it does not give bounds for the graded Betti numbers of the linear strand (i.e., for $j=1$). This seems a natural drawback of our approach, since our key ingredient is the concentration of monomials of degree 2 in the lex ideal of $I_\Delta+(\Theta)$ (cf., \Cref{eq:h2}). However, it follows from the next lemma, that there is no better bound in terms of the total number of vertices $n$ and the dimension $d-1$ than in the standard (non-balanced) Cohen-Macaulay case. More precisely, for any $n$ and any $d$ we construct a balanced Cohen-Macaulay complex whose graded Betti numbers equal $\beta_{i,i+j}(S/\mathfrak{m}^j)$ for $j=1$ and for every $i>0$, where $S=\FF[x_1,\dots,x_{n-d}]$.

\begin{lemma}\label{bound_attained}
Let $n$ and $d$ be positive integers. Let $\Gamma_{n-d+1}$ denote the simplicial complex consisting of the isolated vertices $1,2,\ldots,n-d+1$ and let $\Delta_{d-2}$ be the $(d-2)$-simplex with vertices $\{n-d+2,\ldots,n\}$. Then $\Delta_{d-2}\ast\Gamma_{n-d+1}$ is a balanced $(d-1)$-dimensional Cohen-Macaulay complex. Moreover
	$$\beta_{i,i+1}(\mathbb{F}[\Delta_{d-2}\ast\Gamma_{n-d+1}])= i\binom{n-d+1}{i+1} \quad \mbox{for all } i.$$
\end{lemma}
\begin{proof}
We set $\Delta=\Delta_{d-2}\ast\Gamma_{n-d+1}$. 
As $\Delta$ is the join of a $(d-2)$-dimensional and a $0$-dimensional Cohen-Macaulay complex, it is Cohen-Macaulay of dimension $d-1$. Moreover, coloring the vertices of $\Delta_{d-2}$ with the colors $1,\ldots, d-1$ and assigning color $d$ to all vertices of $\Gamma_{n-d+1}$ gives a proper $d$-coloring of $\Delta$, i.e., $\Delta$ is balanced.

	By Hochster's formula (\Cref{Hochster}), the graded Betti numbers $\beta_{i,i+1}(\mathbb{F}[\Delta])$ are given by
	\begin{equation}
	\label{eq:linearStrand}
	\beta_{i,i+1}(\mathbb{F}[\Delta])=\displaystyle\sum_{W\subseteq [n]:\left| W\right| =i+1} \dim_{\mathbb{F}}\widetilde{H}_{0}(\Delta_W;\mathbb{F}).
	\end{equation}
	As $\Delta_W=(\Delta_{d-2})_W\ast (\Gamma_{n-d+1})_W$, the induced complex $\Delta_W$ is connected whenever $W\cap\{n-d+2,\ldots,n\}\neq\emptyset$. Hence the only non-trivial contributions to \eqref{eq:linearStrand} come from $(i+1)$-element subsets of $[n-d+1]$. For such a subset $W$, the complex $\Delta_W$ consists of $i$ connected components and since there are $\binom{n-d+1}{i+1}$ many such sets, the claim follows.
\end{proof}	

Though we have just seen that Betti numbers (in the linear strand) of balanced Cohen-Macaulay complexes can be as big as the ones for general Cohen-Macaulay complexes, it should also be noted that the simplicial complex $\Delta_{d-2}\ast\Gamma_{n-d+1}$ is special, in the sense that all but one ``big'' color classes are singletons. It is therefore natural to ask, if there are better bounds than those for the general Cohen-Macaulay situation, that take into account the size of the color classes.

\section{A second bound in the Cohen-Macaulay case via lex-plus-squares ideals}\label{sect:SecondBound}

The aim of this section is to provide further upper bounds for the graded Betti numbers of balanced Cohen-Macaulay complexes. On the one hand, those bounds will be a further improvement of the ones from \Cref{bound_precise}. On the other hand, however, they are slightly more complicated to state.  
Our approach is similar to the one used in \Cref{bound_precise} with \emph{lex-plus-squares} ideals as an additional ingredient.  More precisely, we will prove upper bounds for the graded Betti numbers of Artinian quotients $S/I$, where $I\subseteq S$ is a homogeneous ideal having \emph{many} generators in degree $2$, including the squares of the variables $x_1^2,\ldots,x_n^2$. The desired bound for balanced Cohen-Macaulay complexes is then merely an easy application of those more general results.

\subsection{Ideals containing the squares $x_1^2,\ldots,x_n^2$ with many degree $2$ generators}

We recall some necessary definitions and results. As in the previous sections, we let $S=\FF[x_1,\ldots,x_n]$. We further let $P:=(x_1^2,\ldots,x_n^2)\subseteq S$. A monomial ideal $L\subseteq S$ is called \emph{squarefree lex ideal} if for every squarefree monomial $u\in L$ and every monomial $v\in S$ with $\deg(u)=\deg(v)$ and $v>_{\lex} u$ it follows that $v\in L$. For homogeneous ideals containing the squares of the variables the following analog of \Cref{lemma:Bigatti} was shown by Mermin, Peeva and Stillman \cite{MerPeeSti} in characteristic $0$ and by Mermin and Murai \cite{MerminMurai} in arbitrary characteristic:

	\begin{theorem}\label{lppBound}
		Let $I\subseteq S=\FF[x_1,\dots,x_n]$ be a homogeneous ideal containing $P$. Let $I^{\sqlex}\subseteq S$ be the squarefree lex ideal such that $I$ and $I^{\sqlex}+P$ have the same Hilbert function. Then 
		\begin{equation}\label{eq:BettiSquareFreeLex}
		\beta_{i,i+j}^S(S/I)\leq\beta_{i,i+j}^{S}(S/(I^{\sqlex}+P)),
		\end{equation}
		 for all $i,j\geq0$.
	\end{theorem}
	The existence of a squarefree lex ideal $I^{\sqlex}$ as in the previous theorem is a straight-forward consequence of the Clements-Lindstr\"om Theorem \cite{CL}. Moreover, \Cref{lppBound} provides an instance for which the so-called \emph{lex-plus-powers Conjecture} is known to be true (see \cite{EvansRichert}, \cite{Francisco}, \cite{FranciscoRichert} for more details on this topic). 
	
	An ideal of the form $I^{\sqlex}+P$ is called \emph{lex-plus-squares} ideal. 
	It was shown in \cite[Theorem 2.1 and Lemma 3.1(2)]{MerPeeSti} that the graded Betti numbers of ideals of the form $I+P\subseteq S$, where $I\subseteq S$ is a squarefree monomial ideal can be computed via the Betti numbers of \emph{smaller} squarefree monomial ideals, via iterated mapping cones. In the next result, we use $\binom{[n]}{k}$ to denote the set of $k$-element subsets of $[n]$. 
	\begin{proposition}\label{prop:recursive}
	Let $I\subseteq S$ be a squarefree monomial ideal. Then
	\begin{itemize}
\item[(i)] $$\displaystyle\beta_{i,i+j}^S(S/(I+P))=\sum_{k=0}^{j}\left( \sum_{F\in\binom{[n]}{k}}\beta_{i-k,i+j-2k}^S(S/(I:x_F))\right),$$
				where $x_F=\prod_{f\in F}x_f$. 
\item[(ii)] If $I$ is squarefree lex, then the ideal $(I^{\sqlex}:x_F)$ is a squarefree lex ideal in $S_F=S/(x_f~:~ f\in F)$ for any $F\in\binom{[n]}{k}$ .	
\end{itemize}			
	\end{proposition}
	
We have the  following analog of \Cref{p and q} in the squarefree setting.
	
	\begin{lemma}
	\label{lemma:p and q squarefree}
	Let $n\in \NN$ be a positive integer and let $b<\binom{n}{2}$. Let $x_px_q$  be the $b$\textsuperscript{th} largest monomial in the lexicographic order of degree $2$ squarefree monomials in variables $x_1,\ldots,x_n$ and assume $p< q$. Then:
	 $$p=n-1+\left\lfloor\frac{1}{2}-\frac{\sqrt{4n(n-1)-8b+1}}{2}\right\rfloor,$$
		and
		$$q=b+\binom{p+1}{2}-(p-1)n.$$
	\end{lemma}

The proof is deferred to the appendix since it is technical and the precise statement is not needed during the remaining part of this article.

For squarefree lex ideals (or more generally squarefree stable ideals) the following analog of the Eliahou-Kervaire formula \Cref{lemma:eliker} is well-known:

\begin{lemma}\cite[Corollary 7.4.2]{HH}\label{lemma:EK-squarefree}
Let $I\subseteq S$ be a squarefree lex ideal. Then:
 \begin{equation}
 \displaystyle\beta_{i,i+j}^S(S/I)=\sum_{u\in G(I)_{j+1}}\binom{\max(u)-j-1}{i-1},
\end{equation}
			for every $i\geq 1$, $j\geq 0$.
\end{lemma}

%For a monomial ideal $I\subseteq S$, we denote by $I_{\sqfree}$ its \emph{squarefree part}, i.e., the ideal generated by the squarefree minimal generators of $I$. E.g., $\m^\ell_{\sqfree}$ is the ideal generated by all squarefree monomials in degree $\ell$. As this ideal is squarefree lex, we can use \Cref{lemma:EK-squarefree} to compute its graded Betti numbers:
%
%\begin{lemma}\label{lemma:Betti squarefree}
%Let $0\leq\ell\leq n-1$ be a non-negative integer. Then
%\begin{align*}
%\beta_{i,i+j}(S/\m_{\sqfree}^{\ell+1})=\begin{cases}
%1 \quad &\mbox{if } i=j=0,\\
%\binom{n}{i+\ell}\binom{i+\ell-1}{\ell}\quad &\mbox{if } j=\ell,\\
%0 \quad &\mbox{otherwise.} 
%\end{cases}
 %\end{align*}
%\end{lemma}
%
%\begin{proof}
%As  $\mathfrak{m}^{\ell+1}_{\sqfree}$ is generated in degree $\ell+1$, the first and third case of the above statement are immediate. For the second statement, we note that, for fixed $r$, there are $\binom{r-1}{\ell}$ monomials $u$ of degree $\ell+1$ in $\m_{\sqfree}^{\ell+1}$ with $\max(u)=r$. It now follows from \Cref{lemma:EK-squarefree} that 
%\begin{align*}
%\beta_{i,i+\ell}(S/\m^{\ell+1}_{\sqfree})=&\sum_{r=\ell+1}^{n}\binom{r-1}{\ell}\binom{r-\ell-1}{i-1}\\
%=&\sum_{r=\ell+1}^n\frac{(r-1)!}{(r-\ell-1)!(i-1)!\ell!}\\
%=&\binom{\ell+i-1}{\ell}\sum_{r=0}^{n-1}\binom{r}{\ell+i-1}=\binom{n}{\ell+1}\binom{\ell+1-i}{\ell}.
%\end{align*}
%\end{proof}

We can now formulate the main result of this section:

\begin{theorem}\label{thm:BoundsBettiLexPlusSquares}
Let $I\subseteq S$ be a homogeneous ideal not containing any linear form. Let $\dim_\FF(S/(I+P))_2\leq \binom{n}{2}-b$ for some positive integer $b$. Let $x_px_q$, where $p<q$, be the $b$\textsuperscript{th} largest squarefree monomial in $S$ of degree $2$ in lexicographic order. Then:
\begin{align*}
\label{eq:BettiLexPlusSquares}
\beta_{i,i+j}(S/(I+P))\leq \sum_{k=0}^{j-1}\Bigg[&\binom{n-p}{k}\sum_{\ell=p+j-k+1}^{n-k}\binom{\ell-p-1}{j-k}\binom{\ell-j+k-1}{i-k-1}\\
+&\binom{n-q}{k}\sum_{\ell=q+j-k}^{n-k}\binom{\ell-q-1}{j-k-1}\binom{\ell-j+k-1}{i-k-1}\\
+&\binom{n-q}{k-1}\sum_{\ell=q+j-k}^{n-k}\binom{\ell-q}{j-k}\binom{\ell-j+k-1}{i-k-1}\Bigg]\\
+&\binom{n-j}{i-j}\left(\binom{n-p}{j}+\binom{n-q}{j-1}\right)%-\binom{n-q}{j}+\binom{n-q+1}{j}\right).
\end{align*}
for all $i> 0$, $j\geq 2$.
\end{theorem}

\begin{proof}
By \Cref{lppBound} we have $\beta_{i,i+j}(S/(I+P))\leq \beta_{i,i+j}(S/(L+P))$, where $L\subseteq S$ is the squarefree lex ideal such that $L+P$ and $I+P$ have the same Hilbert function. By assumption, $L$ does not contain variables and $\dim_\FF L_2\geq b$. Hence, $L$ contains all squarefree degree $2$ monomials that are lexicographically larger or equal to $x_px_q$. We can further compute $\beta_{i,i+j}(S/(L+P))$ using \Cref{prop:recursive}. For this, we need to analyze the ideals $(L:x_F)$, where $F\in \binom{[n]}{k}$.  
We distinguish four cases (having several subcases):\\

{\sf Case 1:} Assume that $F=\{f\}$ for $1\leq f<p$. In particular, we have $p>1$. Since $L$ is squarefree lex and $x_px_q\in L$, it holds that $x_f x_\ell\in L$ for all $\ell\in [n]\setminus \{f\}$. This implies $(x_i~:~i\in[n]\setminus \{f\})\subseteq (L:x_F)$. As, by \Cref{prop:recursive} (ii) $(L:x_F)$ can be considered as an ideal in $S_F$ and hence no minimal generator is divisible by $x_f$, we infer that $(L:x_F)=(x_i~:~i\in[n]\setminus \{f\})$. As $(L:x_F)$ and $(x_1,\ldots,x_{n-1})$ have the same graded Betti numbers, it follows from \Cref{lemma:EK-squarefree} that $F$ only contributes to $\beta_{i,i+j}(S/(L+P))$ if $j=1$, a case which we do not consider.
%in which case the contribution would equal $\sum_{\ell=1}^{n-1}\binom{\ell-1}{i-2}=\binom{n-1}{i-1}$. 

{\sf Case 2:} Assume that there exist $1\leq s<t\leq n$ such that $\{s,t\}\subseteq F$ and $x_sx_t\geq_\lex x_px_q$. As $L$ is squarefree lex and $x_px_q\in L$, we infer that $x_sx_t\in L$ and hence $1\in (L:x_F)$, i.e., $(L:x_F)=S$. In particular, such $F$ never contributes to $\beta_{i,i+j}(S/(L+P)) $.\\
 
{\sf Case 3:} Suppose that there do not exist $s,t\in F$ ($s\neq t$) with $x_sx_t\geq_\lex x_px_q$. We then have to consider the following two subcases:
\begin{itemize}
\item[$\empty$] {\sf Case 3.1:} $f> p$ for all $f\in F$.
\item[$\empty$] {\sf Case 3.2:} $p\in F$ and $f>q$ for all $f\in F\setminus \{p\}$. 
\end{itemize}
{\sf Case 3.1 (a):} Assume in addition that there exists $f\in F$ with $p<f\leq q$. As $x_px_q\in L$, $x_\ell x_f\geq_\lex x_px_q$ for $1\leq \ell\leq p$ and since $L$ is squarefree lex, we infer that $(x_1,\ldots,x_p)\subseteq (L:x_F)$. Moreover, by \Cref{prop:recursive} (ii) $(L:x_F)$ is squarefree lex as an ideal in $S_F$. If we reorder (and relabel) the variables $x_1,\ldots,x_n$ by first ordering $\{x_i~:~i\notin F\}$ from largest to smallest by increasing indices and then adding $\{x_f~:~f\in F\}$ in any order, the ideal $(L:x_F)$ will be a squarefree lex ideal in $S$ with respect to this ordering of the variables. If $j\neq k$, then, using \Cref{lemma:EK-squarefree}, we conclude
\begin{align*}
\beta_{i-k,i+j-2k}(S/(L:x_F))=&\sum_{\ell=p+j-k+1}^{n-k}\left(\sum_{u\in G(L:x_F)_{j-k+1}}\binom{\ell-(j-k)-1}{i-k-1}\right)\\
\leq &\sum_{\ell=p+j-k+1}^{n-k} \binom{\ell-p-1}{j-k}\binom{\ell-j+k-1}{i-k-1},
\end{align*}
where the last inequality follows from the fact that $G(L:x_F)_{j-k+1}\subseteq G((x_{p+1},\ldots,x_{n-k})^{j-k+1})$. For $j=k$, we note that (after relabeling) we have $G(L:x_F)_1\subseteq (x_1,\ldots,x_{n-k})$, from which it follows that $F$ contributes to $\beta_{i,i+j}(S/(L+P))$ with at most
$$
\sum_{\ell=1}^{n-j}\binom{\ell-1}{i-j-1}=\binom{n-j}{i-j}.
$$

{\sf Case 3.1 (b):} Now suppose that $f>q$ for all $f\in F$. As $F\neq \emptyset$, such $f$ exists. If $p>1$, then, as $L$ is squarefree lex and $x_px_q\in L$, we have $x_{\ell}x_f \in L$ for all $1\leq \ell \leq p-1$. It follows that $x_F\cdot x_{\ell}=x_{F\setminus\{f\}}\cdot (x_\ell\cdot x_f)\in L$ for $1\leq \ell\leq p-1$, which implies $(x_1,\ldots,x_{p-1})\subseteq (L:x_F)$. Moreover, for any $p$, as $x_px_q\in L$, we also have $x_px_\ell\in (L:x_F)$ for $p+1\leq \ell\leq q$. Similar as in Case 3.1 (a) we can assume that, after reordering (and relabeling) the variables, $(L:x_F)$ is a squarefree lex ideal in $S$. As the order of $x_1,\ldots,x_q$ is not affected by this reordering, the previous discussion implies 
\begin{align*}
G(L:x_F)_{j-k+1}\subseteq &\{u\in\Mon_{j-k+1}(x_{p+1},\ldots,x_{n-k})~:~u \mbox{ squarefree}\}\cup\\
& \{x_pu~:u\in \Mon_{j-k}(x_{q+1},\ldots,x_{n-k}),\;u\mbox{ squarefree}\}
\end{align*}
if $j\neq k$.  Using \Cref{lemma:EK-squarefree} we thus obtain
\begin{align*}
\beta_{i-k,i+j-2k}(S/(L:x_F))\leq \sum_{\ell=p+1+j-k}^{n-k}\binom{\ell-1-p}{j-k}\binom{\ell-j+k-1}{i-k-1}+\sum_{\ell=q+j-k}^{n-k}\binom{\ell-1-q}{j-k-1}\binom{\ell-j+k-1}{i-k-1}
\end{align*}
if $j\neq k$. 
For $j=k$, a similar computation as in Case 3.1 (a) shows that $F$ contributes to $\beta_{i,i+j}(S/(L+P))$ with at most $\binom{n-j}{i-j}$.\\

{\sf Case 3.2:} Consider $F\in \binom{[n]}{k}$ such that $p\in F$ and $f>q$ for all $f\in F\setminus\{p\}$. As $x_px_q\in L$ and as $L$ is squarefree lex, it follows that $(x_1,\ldots,x_{p-1},x_{p+1},\ldots,x_q)\subseteq (L:x_F)$. As in Case 3.1, we can assume that after a suitable reordering (and relabeling) of the variables $(L:x_F)$ is a squarefree lex ideal in $S$. (Note that after relabeling $(L:x_F)$ contains $x_1,\ldots,x_{q-1}$.) We infer that
$$
G(I:x_F)_{j-k+1}\subseteq \{u\in\Mon_{j-k+1}(x_q,\ldots,x_{n-k})~:~u\mbox{ squarefree}\}, 
$$
if $j\neq k$ and it hence follows from \Cref{lemma:EK-squarefree} that
\begin{equation*}
\beta_{i-k,i+j-2k}(S/(L:x_F))\leq \sum_{\ell=q+j-k}^{n-k}\binom{\ell-q}{j-k}\binom{\ell-j+k-1}{i-k}
\end{equation*}
if $\neq k$. For $j=k$, it follows from the same arguments as in Case 3.1 (a) that the set $F$ contributes to $\beta_{i,i+j}(S/(L+P))$ with at most $\binom{n-j}{i-j}$.\\

{\sf Case 4:} If $F=\emptyset$, then clearly $(L:x_F)=L$. As $x_px_q\in L$, we obtain that
\begin{align*}
G(L)_{j+1}\subseteq &\{u\in\Mon_{j+1}(x_{p+1},\ldots,x_{n})~:~u \mbox{ squarefree}\}\cup\\
& \{x_pu~:u\in \Mon_{j}(x_{q+1},\ldots,x_{n}),\;u\mbox{ squarefree}\}
\end{align*}
for $j\geq 2$. 
The same computation as in Case 3.1 (b) now yields that
$$
\beta_{i,i+j}(S/(L:x_F))\leq \sum_{\ell=p+1+j}^{n}\binom{\ell-1-p}{j}\binom{\ell-j-1}{i-1}+\sum_{\ell=q+j}^{n}\binom{\ell-1-q}{j-1}\binom{\ell-j-1}{i-1}.
$$

Combining Case 1--4, we finally obtain for $i>0$ and $j>1$:
\begin{align*}
&\beta_{i,i+j}(S/(I+P))\leq \beta_{i,i+j}(S/(L+P))\\
=& \underbrace{\binom{n-j}{i-j}\left(\binom{n-p}{j}-\binom{n-q}{j}\right)}_{\text{Case 3.1(a), }j=k}+\underbrace{\binom{n-j}{i-j}\binom{n-q}{j}}_{\text{Case 3.1(b), } j=k}+\underbrace{\binom{n-j}{i-j}\binom{n-q}{j-1}}_{\text{Case 3.2, } j=k}\\
&+\sum_{k=1}^{j-1}\Bigg[\underbrace{\left(\binom{n-p}{k}-\binom{n-q}{k}\right)\sum_{\ell=p+j-k+1}^{n-k}\binom{\ell-p-1}{j-k}\binom{\ell-j+k-1}{i-k-1}}_{\text{Case 3.1(a)}}\\
&+\underbrace{\binom{n-q}{k}\left(\sum_{\ell=p+j-k+1}^{n-k}\binom{\ell-1-p}{j-k}\binom{\ell-j+k-1}{i-k-1}+\sum_{\ell=q+j-k}^{n-k}\binom{\ell-q-1}{j-k-1}\binom{\ell-j+k-1}{i-k-1}\right)}_{\text{Case 3.1(b)}}\\
&+\underbrace{\binom{n-q}{k-1}\sum_{\ell=q+j-k}^{n-k}\binom{\ell-q}{j-k}\binom{\ell-j+k-1}{i-k-1}}_{\text{Cases 3.2}}\Bigg]\\
=&\sum_{k=0}^{j-1}\Bigg[\binom{n-p}{k}\sum_{\ell=p+j-k+1}^{n-k}\binom{\ell-p-1}{j-k}\binom{\ell-j+k-1}{i-k-1}\\
&+\binom{n-q}{k}\sum_{\ell=q+j-k}^{n-k}\binom{\ell-q-1}{j-k-1}\binom{\ell-j+k-1}{i-k-1}+\binom{n-q}{k-1}\sum_{\ell=q+j-k}^{n-k}\binom{\ell-q}{j-k}\binom{\ell-j+k-1}{i-k-1}\Bigg]\\
&+\binom{n-j}{i-j}\left(\binom{n-p}{j}-\binom{n-q}{j}+\binom{n-q+1}{j}\right).
\end{align*}
This completes the proof.
\end{proof}
There might be several ways to simplify the bound of \Cref{thm:BoundsBettiLexPlusSquares} by losing tightness. However, we decided to state it in the best possible form.

	\subsection{Application: Balanced Cohen-Macaulay complexes revisited}
	The aim of this section is to use \Cref{thm:BoundsBettiLexPlusSquares} in order to get bounds for the graded Betti numbers of balanced Cohen-Macaulay complexes. 
	
Our starting point is the following result due to Stanley (see \cite[Chapter III, Proposition 4.3]{Stanley-greenBook} or \cite{St79}):
	\begin{lemma}\label{sta_balanced}
	Let $\Delta$ be a $(d-1)$-dimensional balanced simplicial complex with vertex partition $V=\displaystyle\bigcupdot_{i=1}^d V_i$ and let $\theta_i:=\displaystyle\sum_{v\in V_i}x_v$, for $1\leq i\leq d$. Then:
	\begin{enumerate}
		\item[(i)] $\theta_1,\ldots,\theta_d$ is an l.s.o.p. for $\FF[\Delta]$.
		\item[(ii)] $x_v^2\in I_\Delta+(\theta_1,\ldots,\theta_d)\subseteq \FF[x_v~:~v\in V]$ for all $v\in V$.
	\end{enumerate}
	\end{lemma} 
	 An l.s.o.p. as in the previous lemma is also referred to as a \emph{colored l.s.o.p.} of $\FF[\Delta]$. If $\Delta$ is strongly connected, which is in particular true if $\Delta$ is  Cohen-Macaulay, then a coloring is unique up to permutation and there is just one colored l.s.o.p. of $\FF[\Delta]$. 

An almost immediate application of \Cref{thm:BoundsBettiLexPlusSquares}, combined with \Cref{sta_balanced} (ii) yields the desired bound for the graded Betti numbers of a balanced Cohen-Macaulay complex:

\begin{theorem}\label{thm:bound_preciseSquares}
	Let $\Delta$ be a $(d-1)$-dimensional balanced Cohen-Macaulay complex with vertex partition $V=\displaystyle\bigcupdot_{i=1}^d V_i$. Let $n:=|V|$, $n_i:=\left| V_i\right|$ and $b:=\displaystyle\sum_{i=1}^d\binom{n_i-1}{2}$. Let $x_px_q$ be the $b$\textsuperscript{th} largest squarefree degree $2$ monomial of $\FF[x_1,\ldots,x_{n-d}]$ in lexicographic order with $p\leq q$. Then
\begin{align*}
\beta_{i,i+j}(\FF[\Delta])\leq \sum_{k=0}^{j-1}\Bigg[&\binom{n-d-p}{k}\sum_{\ell=p+j-k+1}^{n-d-k}\binom{\ell-p-1}{j-k}\binom{\ell-j+k-1}{i-k-1}\\
+&\binom{n-d-q}{k}\sum_{\ell=q+j-k}^{n-d-k}\binom{\ell-q-1}{j-k-1}\binom{\ell-j+k-1}{i-k-1}\\
+&\binom{n-d-q}{k-1}\sum_{\ell=q+j-k}^{n-d-k}\binom{\ell-q}{j-k}\binom{\ell-j+k-1}{i-k-1}\Bigg]\\
+&\binom{n-d-j}{i-j}\left( \binom{n-d-p}{j}+\binom{n-d-q}{j-1}\right)
\end{align*}
	for all $i>0$, $j>1$.
\end{theorem}

\begin{proof}
The proof follows exactly along the same arguments as the one of \Cref{bound_precise}, using the colored l.s.o.p. of $\FF[\Delta]$. By \Cref{sta_balanced} it then holds that the ideal $(\Theta)+I_\Delta$ contains the squares of the variables. It remains to observe that under the isomorphism $\FF[x_1,\ldots,x_n]/(\Theta)\cong R$, the ideal $P=(x_1^2,\ldots,x_n^2)\subseteq \FF[x_1,\ldots,x_n]$ is mapped to a homogeneous ideal containing $(x_1^2,\ldots,x_{n-d}^2)$ and thus $\FF[\Delta]/\Theta\cong R/(I+P)$ for a homogeneous ideal $I\subseteq R$ (not containing linear forms). We further observe that
$$
\dim_\FF (R/(I+P))_2=h_2(\Delta)\leq \binom{n-d+1}{2}-\sum_{i=1}^{d}\binom{n_i}{2}=\binom{n-d}{2}-\sum_{i=1}^d\binom{n_i-1}{2}.
$$
The claim now follows from \Cref{thm:BoundsBettiLexPlusSquares}.
\end{proof}

	\begin{example}
	As in \Cref{example:bound}, we consider $3$-dimensional balanced Cohen-Macaulay complexes with $3$ vertices in each color class, i.e., $n_i=3$ for $1\leq i\leq 4$. We have $b:=\sum_{i=1}^4\binom{3}{2}-8=4$ and $x_1x_5$ is the $4$\textsuperscript{th} largest monomial of degree $2$ in variables $x_1,\ldots,x_8$. The bounds from \Cref{thm:bound_preciseSquares} are recorded in the following table:
	\begin{table}[H]
		\[\begin{array}{r|l|l|l|l|l|l|l|l|c}
		j\setminus i & 0 &1 & 2 & 3 & 4 & 5 & 6 & 7 & 8  \\ \hline
		2&0&38&292&827&1249&1125&609&184&24\\
		3&0&36&267&885&1529&1510&877&280&38\\\
		4&0&21&161&533&1024&1145&727&249&36\\
		\end{array}
		\]
			\label{tab:table2}
	\end{table}	
	Comparing those bounds with the ones from \Cref{tab:table1}, we see that the lex-plus-squares approach gives better bounds for all entries of the Betti table. The improvement is more significant in the lower rows of the Betti tables. 
	\end{example}
\begin{remark}
	Consider again a 3-dimensional balanced Cohen-Macaulay complex $\Delta$ on 12 vertices, but with a different color partition, namely $n_1=1$, $n_2=3$, and $n_3=n_4=4$. Then since every facet must contain the unique vertex of color $1$, $\Delta$ is a cone, hence contractible. \Cref{thm:bound_preciseSquares} yields $\beta_{8,12}(\FF[\Delta])=\dim_{\FF}\widetilde{H}_{3}(\Delta;\FF)\leq 35$. This shows that the bound is not necessarily tight.
\end{remark}
	
\section{The linear strand for balanced pseudomanifolds}\label{section:pseudo}
The aim of  this section is to study the linear strand of the minimal graded free resolution of the Stanley-Reisner ring of a balanced normal pseudomanifold. In particular, we will provide upper bounds for the graded Betti numbers in the linear strand. Previously, such bounds have been shown for general (not necessarily balanced) pseudomanifolds by Murai \cite[Lemma 5.6 (ii)]{MUR} and it follows from a result by Hibi and Terai \cite[Corollary 2.3.2]{TH} that they are tight for stacked spheres. 
We start by recalling those results and by introducing some notation.

Let $\Delta$ and $\Gamma$ be $(d-1)$-dimensional pure simplicial complexes and let $F\in \Delta$ and $G\in \Gamma$ be facets, together with a bijection $\varphi:F\to G$. The \emph{connected sum} of $\Delta$ and $\Gamma$ is the simplicial complex obtained from $\Delta\setminus \{F\}\cup\Gamma\setminus \{G\}$ by identifying $v$ with $\varphi(v)$ for all $v\in F$. A \emph{stacked $(d-1)$-sphere} on $n$ vertices is a $(d-1)$-dimensional simplicial complex $\Delta$ obtained via the connected sum of $n-d$ copies of the boundary of the $d$-simplex. The mentioned results of Murai \cite[Lemma 5.6 (ii)]{MUR} and Hibi and Terai \cite[Corollary 2.3.2]{TH} can be summarized as follows: 

\begin{lemma}\label{standardstacked}
Let $d\geq 3$. Let $\Delta$ be a $(d-1)$-dimensional normal pseudomanifold with $n$ vertices. Then:
\begin{equation*}
\beta_{i,i+1}(\mathbb{F}[\Delta])\leq i\binom{n-d}{i+1} \quad \mbox{ for all } i\geq 0.
\end{equation*}\\
Moreover, those bounds are attained if $\Delta$ is a stacked sphere.
\end{lemma}
We remark that, in \cite{TH}, the authors provide explicit formulas not only for the Betti numbers of the linear strand but for \emph{all} graded Betti numbers of a stacked sphere. In particular, it is shown that these numbers only depend on the number of vertices $n$ and the dimension $d-1$. 

In order to prove a balanced analog of the first statement of \Cref{standardstacked}, the following result due to Fogelsanger \cite{Fogelsanger} will be crucial (see also \cite[Section 5]{NovikSwartz}). 
	
\begin{lemma}\label{fogel}
	Let $d\geq 3$. Let $\Delta$ be a $(d-1)$-dimensional normal pseudomanifold. Then there exist linear forms $\theta_1,\dots,\theta_{d+1}$ such that the multiplication map 
	$$\times \theta_i:\;\left( \mathbb{F}[\Delta]/(\theta_1,\dots,\theta_{i-1})\mathbb{F}[\Delta]\right)_1\longrightarrow\left( \mathbb{F}[\Delta]/(\theta_1,\dots,\theta_{i-1})\mathbb{F}[\Delta]\right)_2$$
	is injective for all $1\leq i\leq d+1$.
\end{lemma}
 Intuitively, the previous result compensates the lack of a regular sequence for normal pseudomanifolds in small degrees, since those need not to be Cohen-Macaulay. 
 
 Recall that a key step for the proofs of \Cref{thm:BettiCM} and \Cref{thm:BoundsBettiLexPlusSquares} was to find upper bounds for the number of generators of the lex ideal and the lex-plus-squares ideal, respectively, of degree $\geq 3$. For the proof of our main result in this section we will use a similar strategy, but since we are interested in the linear strand of the minimal free resolution, we rather need to bound the number of degree $2$ generators in a certain lex-ideal. This will be accomplished via the lower bound theorem for balanced normal pseudomanifolds, which was shown by Klee and Novik \cite[Theorem 3.4]{KN} (see also \cite[Theorem 5.3]{GKN} and \cite[Theorem 4.1]{BrowKlee} for the corresponding result for balanced spheres respectively manifolds and Buchsbaum* complexes).

\begin{theorem}\label{balancedlb}
Let $d\geq 3$ and let $\Delta$ be a $(d-1)$-dimensional balanced normal pseudomanifold. Then 
	$$h_2(\Delta)\geq\frac{d-1}{2}h_1(\Delta).$$
\end{theorem}

We can now state the main result of this section:
\begin{theorem}\label{thm:pseudomanifold}
Let $d\geq 3$ and let $\Delta$ be a $(d-1)$-dimensional balanced normal pseudomanifold on $n$ vertices. Let $b:= \frac{(n-d)(n-2d+2)}{2}$ and let $x_px_q$ (where $p\leq q$) be the $b$\textsuperscript{th} largest degree $2$ monomial of $\FF[x_1,\ldots,x_{n-d-1}]$ in lexicographic order. Then 
	\begin{equation}
	\label{eq:pseudomanifold}
		\beta_{i,i+1}(\mathbb{F}[\Delta])\leq (p-1)\binom{n-d-1}{i}-\binom{p}{i+1}+\binom{q}{i}.
	\end{equation}
\end{theorem}

\begin{proof}
Let $R':=\FF[x_1,\ldots,x_{n-d-1}]$ and let $\Theta=\{\theta_1,\ldots,\theta_{d+1}\}$ be linear forms given by \Cref{fogel}. Then, as in the proof of \Cref{bound_precise}, we let $J\subseteq R$ be the homogeneous ideal with $\mathbb{F}[\Delta]/\Theta\mathbb{F}[\Delta]\cong R/J$ and we let $J^\lex\subseteq R$ be the lex ideal of $J$.  
Using \Cref{fogel}, \Cref{technicalbetti} and \Cref{lemma:Bigatti} we conclude 
\begin{equation*}
\beta_{i,i+1}(\mathbb{F}[\Delta])\leq\beta_{i,i+1}^{S/\Theta S }(\mathbb{F}[\Delta]/\Theta \mathbb{F}[\Delta])= \beta_{i,i+1}^{R}(R/J)\leq \beta_{i,i+1}^{R}(R/J^{\lex}).
\end{equation*}
To prove inequality \eqref{eq:pseudomanifold} we will compute upper bounds for $\beta_{i,i+1}^{R}(R/J^{\lex})$ using \Cref{lemma:eliker}. For those we need an upper bound for the number of generators of degree $2$ in $J^\lex$. More precisely, we will prove the following claim:\\
{\sf Claim:} $\dim_\FF(J^\lex)_2\leq b$.\\
By the definition of the ideals $J$ and $J^\lex$ we have
\begin{align*}
\dim_{\mathbb{F}}(R/J^{\lex})_2&=\dim_{\mathbb{F}}(\mathbb{F}[\Delta]/\Theta\mathbb{F}[\Delta])_2\\
&=h_2(\Delta)-h_1(\Delta)\\
&\geq \frac{d-1}{2}h_1(\Delta)-h_1(\Delta)\\
&=\frac{d-3}{2}(n-d).
\end{align*}
Here, the second equality follows from the injectivity of the multiplication maps in \Cref{fogel} and the inequality holds by \Cref{balancedlb}. We conclude
$$\dim_{\mathbb{F}}(J^{\lex})_2\leq \binom{n-d}{2}-\frac{d-3}{2}(n-d)=\frac{(n-d)(n-2d+2)}{2}=b,$$
which shows the claim.\\
	Since $\dim_{\mathbb{F}}(R/J^{\lex})_1=n-d-1=\dim_{\mathbb{F}}(R)_1$, the ideal $J^\lex$ does not contain variables. Using the just proven claim we conclude that $G(J^{\lex})_2$ contains at most the $b$ lexicographically largest degree $2$ monomials of $R$, i.e.,
	$$
	G(J^\lex)_2\subseteq \{u\in \Mon_2(R)~:~u\geq_{\lex} x_px_q\}.
	$$
To simplify notation, we set $M:=\{u\in \Mon_2(R)~:~u\geq_{\lex} x_px_q\}$. Using \Cref{lemma:eliker}, we infer:
	 \begin{align*}
	 	\beta_{i,i+1}^{R}(R/J^{\lex})\leq&\displaystyle\sum_{ u\in M}\binom{\max(u)-1}{i-1}\\
=&\displaystyle\sum_{\ell=1}^{p}\sum_{\substack{u\in M\\ \max(u)=\ell}}\binom{\ell-1}{i-1}+\displaystyle\sum_{\ell=p+1}^{q}\sum_{\substack{u\in M\\
	 			\max(u)=\ell}}\binom{\ell-1}{i-1}+\\
	 	&+\displaystyle\sum_{\ell=q+1}^{n-d-1}\sum_{\substack{u\in M\\
	 			\max(u)=\ell}}\binom{\ell-1}{i-1}\\
	 	=& \sum_{\ell=1}^{p}\ell\binom{\ell-1}{i-1}+p\sum_{\ell=p+1}^{q}\binom{\ell-1}{i-1}+(p-1)\sum_{\ell=q+1}^{n-d-1}\binom{\ell-1}{i-1}\\
	 	=& i\binom{p+1}{i+1}+(p-1)\binom{n-d-1}{i}-p\binom{p}{i}+\binom{q}{i}\\
	 	=& (p-1)\binom{n-d-1}{i}-\binom{p}{i+1}+\binom{q}{i}
	 	\end{align*}
	 	for all $i\geq 0$. This finishes the proof.
	 \end{proof}
	 
Note that, unlike the bounds from \Cref{thm:BettiCM} and \Cref{thm:bound_preciseSquares}, the bounds from \Cref{thm:pseudomanifold} do not depend on the sizes of the color classes.

\begin{example}
Let $\Delta$ be a $3$-dimensional balanced pseudomanifold on $12$ vertices, with an arbitrary partition of the vertices into color classes. We have $b=\frac{(n-d)(n-2d+2)}{2}=24$ and $x_5x_6$ is the $24$\textsuperscript{th} largest degree $2$ monomial in variables $x_1,\ldots,x_7$. The bounds for $\beta_{i,i+1}(\FF[\Delta])$ provided by  \Cref{thm:pseudomanifold} are recorded in the following table.
\begin{table}[H]
	\[\begin{array}{r|l|l|l|l|l|l|l|l|c}
	j\setminus i & 0 &1 & 2 & 3 & 4 & 5 & 6 & 7 & 8  \\ \hline
	1&0&24&89&155&154&90&29&4&0\\
	\end{array}
	\]
\end{table}
One should compare those with the bounds provided by \Cref{standardstacked} for arbitrary (not necessarily balanced) pseudomanifolds:
\begin{table}[H]
	\[\begin{array}{r|l|l|l|l|l|l|l|l|c}
	j\setminus i & 0 &1 & 2 & 3 & 4 & 5 & 6 & 7 & 8  \\ \hline
	1&0&28&112&210&224&140&48&7&0\\
	\end{array}
	\]
	\end{table}
While the bounds in the previous table are realized by any stacked $3$-sphere on $12$ vertices, we do not know if the ones for the balanced case, shown in the upper table, are attained. In the next section we will see that they are not attained by the balanced analog of stacked spheres.
\end{example}
\begin{remark}
In view of \Cref{thm:pseudomanifold} a natural question that arises is if one can also bound the entries of the $j$\textsuperscript{th} row of the Betti table of a balanced pseudomanifold for $j\geq 2$. In order for our approach to work, this would require the multiplication maps from \Cref{fogel} to be injective also for higher degrees; a property that is closely related to Lefschetz properties. 
\end{remark}

\section{Betti numbers of stacked cross-polytopal spheres}\label{sectionstacked}
The aim of this section is to compute the graded Betti numbers of stacked cross-polytopal spheres explicitly. Stacked cross-polytopal spheres can be considered as the balanced analog of stacked spheres, in the sense that both minimize the $h$-vector among the class of balanced normal pseudomanifolds respectively all normal pseudomanifolds (see \cite[Theorem 4.1]{KN} and e.g., \cite{Fogelsanger,Kalai,tay}). For stacked spheres, explicit formulas for their graded Betti numbers were provided by Hibi and Terai \cite{TH} and it was shown that they only depend on the number of vertices and the dimension but not on the combinatorial type of the stacked sphere (see also \Cref{standardstacked}). 

We start by introducing some necessary definitions. We denote the boundary complex of the $d$-dimensional cross-polytope by $\C_d$. Combinatorially, $\C_d$ is given as the join of $d$ pairs of disconnected vertices, i.e.,
\begin{equation*}
\label{eq:cross-polytope}
\C_d:=\{v_1,w_1\}\ast\cdots \ast \{v_d,w_d\}.
\end{equation*}

\begin{definition}
	Let $n=kd$ for some integer $k\geq2$. A \emph{stacked cross-polytopal $(d-1)$-sphere} on $n$ vertices is a simplicial complex obtained via the connected sum of $k-1$ copies of $\mathcal{C}_d$. 
	We denote by $\mathcal{ST}^{\times}(n,d)$ the set of all stacked cross-polytopal $(d-1)$-spheres on $n$ vertices.
\end{definition}
Observe that $\ST^{\times}(2d,d)=\{\C_d\}$, and, as $\C_d$ is balanced, so is any stacked cross-polytopal sphere. In analogy with the non-balanced setting, for $k\geq 4$, there exist stacked cross-polytopal spheres in $\mathcal{ST}^{\times}(kd,d)$ of different combinatorial types, as depicted in \Cref{three_stacked}. Nevertheless, it is easily seen that the $f$-vector of a stacked cross-polytopal only depends on $n$ and $d$. In this section, we will show the same behavior for their graded Betti numbers.
\begin{figure}[H]
	\centering
	\includegraphics[scale=0.5]{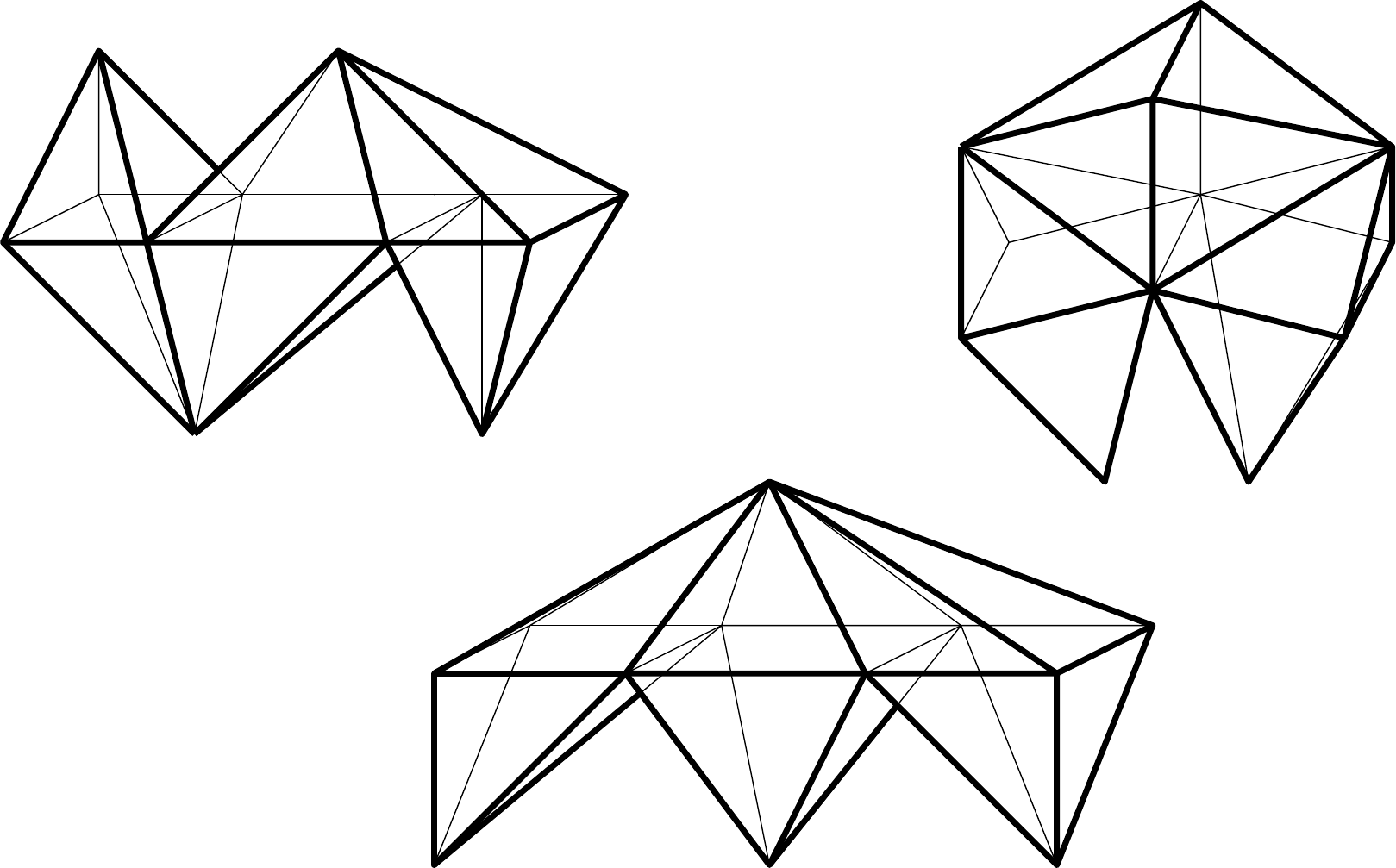}
	\caption{Three non simplicially isomorphic spheres in $\mathcal{ST}^{\times}(12,3)$.}
	\label{three_stacked}
\end{figure} 

As a warm-up, we compute the Betti numbers of the boundary complex of the cross-polytope.
\begin{lemma}\label{lemma:cross}
Let $d\geq 1$. Then $\beta_{i,i+j}(\FF[\C_d])=0$ for all $i\geq 0$ and $j\neq i$. Moreover,
\begin{equation*}
\beta_{i,2i}(\FF[\C_d])=\binom{d}{i}
\end{equation*}
for all $i$. 
\end{lemma}
\begin{proof}
Being generated by $d$ pairwise coprime monomials, the Stanley-Reisner ideal of $\C_d$ is a complete intersection, and hence it is minimally resolved by the Koszul complex. 
\end{proof}
%%%%%%%%%%%%%%kann evtl wegfallen
%Consider $\Delta\in \ST^\times(kd,d)$ and let $\Delta_1,\ldots,\Delta_{k-1}$ denote the copies of $\C_d$ from which $\Delta$ was constructed. If $d\geq 3$, then the missing facets of $\Delta$ (i.e., non-faces of cardinality $d$) are exactly the facets that have been identified (and removed) during the glueing process. Moreover, if $k\geq3$, $i\geq 2$ and $F$ is the missing facet that comes from glueing $\Delta_i$ to $\Delta_1\#\cdots \Delta_{i-1}$ along $F$, then there is a unique facet $F^{\ext}\in \Delta_i$ with $F\cap F^{\ext}=\emptyset$. We call $F^{\ext}$ the \emph{extremal facet} of  $\Delta_i$ ($i\geq 2$) and $F$ its \emph{opposite} (see \Cref{extremal_opposite} for an example). We further define the extremal facet of $\Delta_1$ to be the unique facet of $\Delta_1$ that is disjoint from the missing facet obtained by glueing $\Delta_2$ to $\Delta_1$. Finally, we will say that a facet $H$ of $\Delta$ is \emph{extremal for $\Delta$} if it is extremal for some $\Delta_i$. If $\Delta$ is a single cross-polytope, then every facet of $\Delta$ will be called \emph{extremal}. We note that extremal facets of $\Delta_i$ might become missing facets of $\Delta$ and in particular opposites of extremal facets. However, any stacked cross-polytopal sphere has at least $2$ extremal facets. 
%%%%%%%%%%%%%%%%%
The following immediate lemma will be very useful, in order to derive a recursive formula for the graded Betti numbers of stacked cross-polytopal spheres.

\begin{lemma}\label{lemma:homology}
Let $d\geq 3$. Let $\Delta\in \ST^\times(kd,d)$ be a stacked cross-polytopal sphere on vertex set $V$ and let $F$ be a facet of $\Delta$. Then for any $W\subseteq V$, 
\begin{equation*}
\widetilde{H}_j(\Delta_W;\FF)=\widetilde{H}_j((\Delta\setminus\{F\})_W;\FF) \quad \mbox{for all } 0\leq j\leq d-3.
\end{equation*}
\end{lemma}
\begin{proof}
The statement is immediate since $\Delta$ and $\Delta\setminus \{F\}$ share the same skeleta up to dimension $d-2$. 
\end{proof}

Consider $\Delta\in \ST^\times(kd,d)$ and let $\Diamond_1,\ldots,\Diamond_{k-1}$ denote the copies of $\C_d$ from which $\Delta$ was constructed. We call a facet $F\in\Delta\cap\Diamond_i$ \emph{extremal} if $V(\Diamond_i)\setminus F\notin \Delta$, and the facet $V(\Diamond_i)\setminus F$ is called the \emph{opposite} of $F$. Intuitively a facet $F$ of $\Delta$ is extremal if removing all the vertices in $F$ from $\Delta$ yields a complex $\Gamma\setminus \{G\}$, where $\Gamma\in\ST^\times((k-1)d,d)$ and $G$ is the opposite of $F$.
\begin{figure}[H]
	\centering
	\includegraphics[scale=0.7]{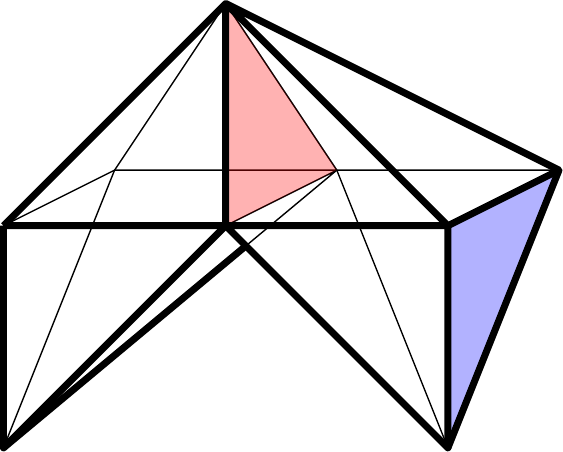}
	\caption{An {\color{blue}{extremal facet}} and its {\color{red}{opposite}}.}
	\label{extremal_opposite}
\end{figure} 

We have the following recursive formulas for Betti numbers of stacked cross-polytopal spheres.
\begin{remark}
	Note that for the case $j=1$ the following formula can be deduced from Corollary 3.4 in \cite{ChKi}. We report its proof anyway, as the idea is analogous to the case $j\geq 2$.  
\end{remark}
\begin{theorem}\label{thm:rec}
Let $n\geq 3d$ and $\Delta\in \ST^\times (n,d)$. Then
\begin{equation}
\beta_{i,i+j}(\FF[\Delta])=\begin{cases}
\displaystyle\sum_{\ell=0}^d\binom{d}{\ell}\beta_{i-\ell,i-\ell+1}(\FF[\Gamma])+d\binom{n-2d}{i-1}+\sum_{\ell=1}^{\min\{i,d\}}\binom{d}{\ell}\binom{n-2d}{i+1-\ell}, &\mbox{ if } j=1\\
\displaystyle\sum_{\ell=0}^{d}\binom{d}{\ell}\beta_{i-\ell,i-\ell+j}(\FF[\Gamma])+\binom{d}{j}\binom{n-2d}{i-j}, &\mbox{ if } 2\leq j\leq d-2,
\end{cases}
\end{equation}
with $\Gamma\in\ST^\times (n-d,d)$. In particular, the graded Betti numbers of $\Delta$ only depend on $n$ and $d$. 
\end{theorem}

%was schreiben, dass man auch für d-1, 1 Formeln direkt hat.
\begin{proof}
We will compute the graded Betti numbers using Hochster's formula. 
Let $V$ be the vertex set of $\Delta$ and let $F$ be an extremal facet of $\Delta$ with opposite $G$. Then, we can write $\Delta=(\Gamma\setminus \{G\})\cup (\Diamond\setminus \{G\})$, where $\Gamma\in \ST^{\times}(n-d,d)$ and $\Diamond$ is the boundary complex of the $d$-dimensional cross-polytope on vertex set $F\cup G$. In particular, $(\Gamma\setminus \{G\})\cap (\Diamond\setminus \{G\})=\partial(G)$. We now distinguish two cases.\\
{\sf Case 1:} $j=1$. Let $W\subseteq V$. We have several cases:
\begin{itemize}
\item[(a)] If $W\subseteq V(\Gamma)$, then $\Delta_W=(\Gamma\setminus \{G\})_W$. By \Cref{lemma:homology}, $(\Gamma\setminus \{G\})_W$ (thus $\Delta_W$) and $\Gamma_W$ have the same number of connected components and hence $\widetilde{H}_0(\Delta_W;\FF)=\widetilde{H}_0(\Gamma_W;\FF)$. 
\item[(b)] If $W\subseteq V(\Diamond)$, then it follows as in (b) that $\widetilde{H}_0(\Delta_W;\FF)=\widetilde{H}_0(\Diamond_W;\FF)$.
\item[(c)] Assume that $W\cap (V(\Gamma)\setminus G)\neq \emptyset$ and $W\cap (V(\Diamond)\setminus G)\neq \emptyset$. Then, $\Delta_W=(\Gamma\setminus\{G\})_W\cup(\Diamond\setminus \{G\})_W$. If, in addition, $W\cap G=\emptyset$, then this union is disjoint and, using \Cref{lemma:homology} we conclude that the number of connected components of $\Delta_W$ equals the sum of the number of connected components of $\Gamma_W$ and $\Diamond_W$. Thus, as neither $\Gamma_W$ nor $\Diamond_W$ is the empty complex, 
$$
\dim_\FF \widetilde{H}_0(\Delta_W;\FF)=\dim_\FF\widetilde{H}_0(\Gamma_W;\FF)+\dim_\FF\widetilde{H}_0(\Diamond_W;\FF)+1.
$$
If $W\cap G\neq \emptyset$, then the number of connected components of $\Delta_W$ is one less than the sum of the number of connected components of $(\Gamma\setminus \{G\})_W$ and $(\Diamond\setminus\{G\})_W$. In particular, using \Cref{lemma:homology}, we infer
$$
\dim_\FF \widetilde{H}_0(\Delta_W;\FF)=\dim_\FF\widetilde{H}_0(\Gamma_W;\FF)+\dim_\FF\widetilde{H}_0(\Diamond_W;\FF).
$$
\end{itemize}
Using Hochster's formula we obtain:
\begin{align*}
\beta_{i,i+1}(\FF[\Delta])=&\sum_{W\subseteq V;\;|W|=i+1}\dim_\FF\widetilde{H}_{i-1}(\Delta_W;\FF)\\
=&\sum_{\substack{W\subseteq V;\;|W|=i+1\\W\cap G\neq\emptyset}}\left(\dim_\FF\widetilde{H}_0(\Gamma_W;\FF)+\dim_\FF\widetilde{H}_0(\Diamond_W;\FF)\right)\\
&+\sum_{\substack{W\subseteq V\setminus G;\;|W|=i+1\\W\cap V(\Gamma)\neq \emptyset; W\cap V(\Diamond)\neq \emptyset}}\left(\dim_\FF\widetilde{H}_0(\Gamma_W;\FF)+\dim_\FF\widetilde{H}_0(\Diamond_W;\FF)+1\right)\\
&+\sum_{\substack{W\subseteq V(\Gamma)\setminus G\\|W|=i+1}}\dim_\FF\widetilde{H}_0(\Gamma_W;\FF)+\sum_{\substack {W\subseteq V(\Diamond)\setminus G\\|W|=i+1}}\dim_\FF\widetilde{H}_0(\Diamond_W;\FF).
\end{align*}
For $W\subseteq V(\Gamma)$ (respectively $W\subseteq V(\Diamond)$) the term $\dim_\FF\widetilde{H}_0(\Gamma_W;\FF)$ (respectively $\dim_\FF\widetilde{H}_0(\Diamond_W;\FF)$) appears $\binom{d}{i+1-|W|}$ (respectively $\binom{n-2d}{i+1-|W|}$) times in the previous expression. Moreover, there are $\sum_{\ell=1}^{\min\{i,d\}}\binom{d}{\ell}\binom{n-2d}{i+1-\ell}$ $(i+1)$-subsets $W$ of $V\setminus G$ with $W\cap V(\Gamma)\neq \emptyset$ and  $W\cap V(\Diamond)\neq \emptyset$.  This implies
\begin{align*}
\beta_{i,i+1}(\FF[\Delta])=&\sum_{\ell=1}^{i+1}\binom{d}{i+1-\ell}\left(\sum_{W\subseteq V(\Gamma),\;|W|=\ell}\dim_\FF\widetilde{H}_0(\Gamma_W;\FF)\right)\\
&+\sum_{\ell=1}^{2d}\binom{n-2d}{i+1-\ell}\left(\sum_{W\subseteq V(\Diamond),\;|W|=\ell}\dim_\FF\widetilde{H}_0(\Diamond;\FF)\right)
+\sum_{\ell=1}^{\min\{i,d\}}\binom{d}{\ell}\binom{n-2d}{i+1-\ell}\\
=&\sum_{\ell=i+1-d}^{i+1}\binom{d}{i+1-\ell}\beta_{\ell-1,\ell}(\FF[\Gamma])+\sum_{\ell=1}^{2d}\binom{n-2d}{i+1-\ell}\beta_{\ell-1,\ell}(\FF[\Diamond])+\sum_{\ell=1}^{\min\{i,d\}}\binom{d}{\ell}\binom{n-2d}{i+1-\ell},
\end{align*}
where the last equality holds by Hochster's formula. The desired recursion for $\beta_{i,i+1}(\FF[\Delta])$ now follows from a simple index shift. \\

{\sf Case 2:} $2\leq j\leq d-2$. Let $W\subseteq V$. We consider two cases.
\begin{itemize}
\item[(a)] If $W\subseteq V(\Gamma)$, then it follows from \Cref{lemma:homology} that 
$$
\widetilde{H}_j(\Delta_W;\FF)=\widetilde{H}_j(\Gamma_W;\FF) \mbox{ for } 0\leq j\leq d-3.
$$
\item[(b)] If $W\subseteq V(\Diamond)$, then it follows as in (a) that
$$
\widetilde{H}_j(\Delta_W;\FF)=\widetilde{H}_j(\Diamond_W;\FF) \mbox{ for } 0\leq j\leq d-3. 
$$
\item[(c)] Assume that $W\cap(V(\Gamma)\setminus G)\neq \emptyset$ and $W\cap (V(\Diamond)\setminus G)\neq \emptyset$. Then, $\Delta_W=(\Gamma\setminus\{G\})_W\cup(\Diamond\setminus \{G\})_W$. Let $1\leq j\leq d-3$. We have the following Mayer-Vietoris exact sequence 
\begin{align}\label{eq:MV}
 \ldots \to \underbrace{\widetilde{H}_{j}(\partial(G)_W;\FF)}_{=0}&\to \widetilde{H}_j((\Gamma\setminus \{G\})_W;\FF)\oplus \widetilde{H}_j((\Diamond\setminus \{G\})_W;\FF)\notag \\
 &\to \widetilde{H}_{j}(\Delta_W;\FF)\to \underbrace{\widetilde{H}_{j-1}(\partial(G)_W;\FF)}_{=0}\to \ldots,
\end{align}
where we use that $(\Gamma\setminus\{G\})_W\cap (\Diamond\setminus \{G\})_W=(\partial(G))_W$, which has always trivial homology in dimension $\leq d-3$.  
It follows from \Cref{eq:MV} combined with \Cref{lemma:homology} that 
$$
\widetilde{H}_{j}(\Delta_W;\FF)\cong\widetilde{H}_j(\Gamma_W;\FF)\oplus \widetilde{H}_j(\Diamond_W;\FF) \quad \mbox{ for } 1\leq j\leq d-3.
$$
Using Hochster's formula we conclude
\begin{align*}
\beta_{i,i+j}(\FF[\Delta])=&\sum_{W\subseteq V,\;|W|=i+1}\left(\dim_\FF\widetilde{H}_{j-1}(\Gamma_W;\FF)+ \dim_\FF\widetilde{H}_{j-1}(\Diamond_W;\FF)\right)\\
=&\sum_{\ell=i+j-d}^{i+j}\binom{d}{i+j-\ell}\left(\sum_{W\subseteq V(\Gamma),\;|W|=\ell}\dim_\FF\widetilde{H}_{j-1}(\Gamma_W;\FF)\right)+\\
&+\sum_{\ell=1}^{2d}\binom{n-2d}{i+j-\ell}\left(\sum_{W\subseteq V(\Diamond),\;|W|=\ell}\dim_\FF\widetilde{H}_{j-1}(\Diamond;\FF)\right)\\
=&\sum_{\ell=i+j-d}^{i+j}\binom{d}{i+j-\ell}\beta_{\ell-j,\ell}(\FF[\Gamma])+\sum_{\ell=1}^{2d}\binom{n-2d}{i+j-\ell}\beta_{\ell-j,\ell}(\FF[\Diamond])\\
=&\sum_{\ell=0}^{d}\binom{d}{\ell}\beta_{i-d+\ell,i-d+\ell+j}(\FF[\Gamma])+\binom{n-2d}{i-j}\binom{d}{j},
\end{align*}
where the second equality follows, as in Case 1, by a simple counting argument and the last equality follows from \Cref{lemma:cross}. 
\end{itemize}
The statement of the ``In particular''-part follows directly by applying the recursion iteratively, and from $\ST^\times(2d,d)=\{\C_d\}$. 
\end{proof}

\begin{remark}
We remark that due to graded Poincar\'e duality the graded Betti numbers of any stacked cross-polytopal sphere $\Delta\in \ST^\times(n,d)$ exhibit the following symmetry:
\begin{equation}\label{eq:PoincareDuality}
\beta_{i,i+j}(\FF[\Delta])=\beta_{n-d-i,n-i-j}(\FF[\Delta]).
\end{equation}
This in particular implies $\beta_{n-d,n}(\FF[\Delta])=1$ and $\beta_{i,i+d}(\FF[\Delta])=0$ for $0\leq i<n-d$. Moreover, also $\beta_{i,i+d-1}(\FF[\Delta])$ can be computed using the recursion from \Cref{thm:rec} (for the linear strand). 
\end{remark}
In order to derive explicit formulas for the graded Betti numbers of a stacked cross-polytopal sphere, we need to convert the recursive formula of \Cref{thm:rec} into a closed expression.

\begin{theorem}\label{thm:BettiCross}
Let $d\geq 3$, $k\geq 2$ and let $\Delta\in \ST^\times(kd,d)$ be a stacked cross-polytopal sphere. Then, $\beta_{0,0}(\FF[\Delta])=\beta_{(k-1)d,kd}(\FF[\Delta])=1$ and for $i\geq 0$:\\
$\beta_{i,i+j}(\FF[\Delta])=\begin{cases}
	\displaystyle(k-2)\binom{d(k-1)}{i+1}-(k-1)\binom{d(k-2)}{i+1}+d(k-1)\binom{d(k-2)}{i-1} & j=1\\
	\displaystyle(k-1)\binom{d}{j}\binom{d(k-2)}{i-j} & 2\leq j\leq d-2\\
	\displaystyle(k-2)\binom{d(k-1)}{i-1}-(k-1)\binom{d(k-2)}{i-d-1}+d(k-1)\binom{d(k-2)}{i-d+1} & j=d-1
\end{cases}$	
\end{theorem}
\begin{proof}
We proof the claim by induction on $k$. \\
For $k=2$, the first line above equals $d$ if $i=1$ and $0$ otherwise. Similarly, the second line equals $\binom{d}{i}$ if $j=i$ and $0$ otherwise. The claim for $k=2$ now follows from \Cref{lemma:cross}.\\
Let $k\geq 3$ and let $\Delta \in \ST^\times(kd,d)$. We first show the case $j=1$.\\
Using \Cref{thm:rec} and then the induction hypothesis, we conclude
	\begin{align}\label{eq:computations}
	\beta_{i,i+1}(\FF[\Delta])=&\sum_{\ell=0}^{\min\{i,d\}}\binom{d}{\ell}\beta_{i-\ell,i-\ell+1}(\FF[\Gamma])+d\binom{n-2d}{i-1}+\sum_{\ell=1}^{\min\{i,d\}}\binom{d}{\ell}\binom{n-2d}{i+1-\ell}\notag\\
	=& (k-3)\sum_{\ell=0}^{\min\{i,d\}}\binom{d}{\ell}\binom{d(k-2)}{(i+1)-\ell}-(k-2)\sum_{\ell=0}^{\min\{i,d\}}\binom{d}{\ell}\binom{d(k-3)}{(i+1)-\ell}\notag \\
	&+d(k-2)\sum_{\ell=0}^{\min\{i,d\}}\binom{d}{\ell}\binom{d(k-3)}{(i-1)-\ell}+d\binom{d(k-2)}{i-1}+\sum_{\ell=1}^{\min\{i,d\}}\binom{d}{\ell}\binom{d(k-2)}{(i+1)-\ell},
	\end{align}
	where $\Gamma\in\ST\times((k-1)d,d)$. 
	We now assume that $\min\{i,d\}=d$. We notice that in \eqref{eq:computations}, we can shift the upper summation indices to $i+1$ in the first $2$ sums and to $i-1$ in the third sum. Using Vandermonde identity we obtain
	\begin{align*} 
	\beta_{i,i+1}(\FF[\Delta])=&(k-3)\binom{d(k-1)}{i+1}-(k-2)\binom{d(k-2)}{i+1}+d(k-2)\binom{d(k-2)}{i-1}\\
	&+d\binom{d(k-2)}{i-1}+\left(\binom{d(k-1)}{i+1}-\binom{d(k-2)}{i+1}\right)\\
	=&(k-2)\binom{d(k-1)}{i+1}-(k-1)\binom{d(k-2)}{i+1}+d(k-1)\binom{d(k-2)}{i-1}.
	\end{align*}
If $i<d$ (thus $\min\{i,d\}=i$), then the same computation as above with an additional summand of $-(k-3)\binom{d}{i+1}$, $(k-2)\binom{d}{i+1}$ and $-\binom{d}{i+1}$ for the first, second and fourth sum, respectively, shows the formula for the first line.\\

We now show the case $1<j\leq d-2$:\\
Applying \Cref{thm:rec} and the induction hypothesis, we obtain
\begin{align*}
\beta_{i,i+j}(\FF[\Delta])=&\sum_{\ell=0}^{\min\{i,d\}}\binom{d}{\ell}\beta_{i-\ell,i-\ell+j}(\FF[\Gamma])+\binom{d}{j}\binom{d(k-2)}{i-j}\\
=&\sum_{\ell=0}^{\min\{i,d\}}\binom{d}{\ell}(k-2)\binom{d}{j}\binom{d(k-3)}{i-j-\ell}+\binom{d}{j}\binom{d(k-2)}{i-j}\\
=&(k-2)\binom{d}{j}\sum_{\ell=0}^{\min\{i-j,d\}}\binom{d}{\ell}\binom{d(k-3)}{i-j-\ell}+\binom{d}{j}\binom{d(k-2)}{i-j}\\
=&(k-2)\binom{d}{j}\binom{d(k-2)}{i-j}+\binom{d}{j}\binom{d(k-2)}{i-j}=(k-1)\binom{d}{j}\binom{d(k-2)}{i-j},
\end{align*}
where $\Gamma\in\ST^\times((k-1)d,d) $ and the fourth equality follows from Vandermonde's identity after observing that shifting the upper index of the sum to $i-j$ does not change the sum. \\

The statement in the last line ($j=d-1$) follows from graded Poincar\'e duality (see \eqref{eq:PoincareDuality}). 
\end{proof}

\begin{example}\label{example:tightness}
For stacked cross-polytopal $3$-spheres on $12$ vertices \Cref{thm:BettiCross} yields the following Betti numbers for the linear strand:
\begin{table}[H]
	\[\begin{array}{r|l|l|l|l|l|l|l|l|c}
	j\setminus i & 0 &1 & 2 & 3 & 4 & 5 & 6 & 7 & 8  \\ \hline
	1&0&24&80&116&88&36&8&1&0\\
	\end{array}.
	\]
\end{table}
If we compare them with the bounds for the Betti numbers of a $3$-dimensional balanced normal pseudomanifold on $12$ vertices from \Cref{thm:pseudomanifold}, displayed in the next table, we see that they are smaller in almost all places.
\begin{table}[H]
	\[\begin{array}{r|l|l|l|l|l|l|l|l|c}
	j\setminus i & 0 &1 & 2 & 3 & 4 & 5 & 6 & 7 & 8  \\ \hline
	1&0&24&89&155&154&90&29&4&0\\
	\end{array},
	\]
\end{table}
\end{example}

In light of \Cref{standardstacked} and the analogy between stacked and cross-polytopal stacked spheres, the previous example suggests the following  conjecture:

\begin{conjecture}
	Let $\Delta$ be a $(d-1)$-dimensional balanced normal pseudomanifold, with $d\geq 4$ and let $f_0(\Delta)=kd$, for some integer $k\geq 2$. Then
	$$\beta_{i,i+1}(\FF[\Delta])\leq \beta_{i,i+1}(\FF[\Gamma]),$$
	for $\Gamma\in\mathcal{ST}^{\times}(kd,d)$, and for every $i\geq 0$.
\end{conjecture}

\section*{Acknowledgement}
We would like to thank Giulio Caviglia for directing us to the results on lex-plus-squares ideals by Mermin, Peeva and Stillman. This led to the content of \Cref{sect:SecondBound}. We are also grateful to Satoshi Murai for pointing out Corollary 1.4 in \cite{HerzogHibi:ComponentwiseLinear}, which simplified and shortened the proof of \Cref{prop:Gotzmann} . 
\bibliographystyle{alpha}
\bibliography{bibliography}
\section{Appendix}
\begin{proof}[Proof of \Cref{p and q}]
	Let $M$ be the $n\times n$ upper triangular matrix obtained by listing the degree $2$ monomials in variables $x_1,\ldots,x_n$ in decreasing lexicographic order from left to right and top to bottom:
	$$	M=\begin{bmatrix}
	x_{1}^2 & x_{1}x_{2}  & \dots  & x_{1}x_{n} \\
	0 & x_{2}^2  & \dots  & x_{2}x_{n} \\
	\vdots & \vdots  & \ddots & \vdots \\
	0 & 0 & \dots  & x_{n}^2
	\end{bmatrix}.$$
	From this ordering, it is easily seen, that, if $x_px_q$ (with $p<q$) is the $b$-th largest degree $2$ monomial in lexicographic order, then 
	$$
	n-p=\max\{s\in \NN~:~\sum_{\ell=1}^{s}\ell\leq\binom{n+1}{2}-b\}. $$
	As $s= -\frac{1}{2}+\frac{\sqrt{4n(n+1)+1-8b}}{2}$ is the unique non-negative solution to the equation
$$	
(s+1)s/2=(n+1)n/2 -b,
$$
we conclude that
$$
p=n-\left\lfloor -\frac{1}{2}+\frac{\sqrt{4n(n+1)+1-8b}}{2}\right\rfloor.
$$
Looking at the matrix $M$, we deduce that the index $q$, (i.e., the column index of $x_px_q$ in $M$) is given by
\begin{align*}
q=&b-\sum_{\ell=1}^{p-1}(n+1-\ell) +(p-1)\\
=& b-(p-1)(n+1)+\frac{p(p-1)}{2}+(p-1)\\
=&b+ \frac{(p-1)(-2-2n+p+2)}{2}=b+\frac{(p-1)(p-2n)}{2}.
\end{align*}
The claim follows.
\end{proof}	
\begin{proof}[Proof of \Cref{lemma:p and q squarefree}]
As in the proof of \Cref{p and q} it is easy to see that, if $x_px_q$ (with $p<q$) is the $b$-th largest squarefree degree $2$ monomial, then 
$$
n-p=\max\{s\in \NN~:~\sum_{\ell=1}^s\ell\leq \binom{n}{2}-b\}+1.
$$
Since $s=-\frac{1}{2}+\frac{\sqrt{4n(n-1)-8b+1}}{2}$ is the  unique  non-negative solution to the equation
$$
(s+1)s/2=n(n-1)/2-b,
$$
we infer that $p= n-1+\left\lfloor\frac{1}{2}-\frac{\sqrt{4n(n-1)-8b+1}}{2}\right\rfloor$. As $q=b-\sum_{\ell=1}^{p-1}(n-\ell)+p$, the claim follows from a straight forward computation. 
\end{proof}
\end{document}